\documentclass[reqno]{amsart}
\usepackage{amsmath, amsthm, amssymb, amstext}

\usepackage{hyperref,xcolor}
\hypersetup{pdfborder={0 0 0},colorlinks}
\usepackage{enumitem}
\allowdisplaybreaks
\usepackage{todonotes}

\newtheorem{theorem}{Theorem}
\newtheorem{remark}[theorem]{Remark}
\newtheorem{lemma}[theorem]{Lemma}
\newtheorem{proposition}[theorem]{Proposition}

\newtheorem{definition}[theorem]{Definition}

\newcommand{\N}{\mathbb{N}}
\newcommand{\R}{\mathbb{R}}
\newcommand{\Ne}{\mathcal{N}}
\newcommand{\dx}{\, {\rm d} x}

\newcommand{\dt}{\, {\rm d} t}

\newcommand{\eps}{\varepsilon}
\renewcommand{\L}{\mathcal{L}}
\renewcommand{\H}{\mathcal{H}}

\numberwithin{theorem}{section}
\numberwithin{equation}{section}

\title[Degenerate singular Kirchhoff problems in Musielak-Orlicz spaces]{Degenerate singular Kirchhoff problems \\ in Musielak-Orlicz spaces}

\author[U. Guarnotta]{Umberto Guarnotta}
\address[U. Guarnotta]{Dipartimento di Ingegneria Industriale e Scienze Matematiche, Università Politecnica delle Marche, Via Brecce Bianche 12, 60131 Ancona, Italy}
\email{u.guarnotta@univpm.it}
\author[P. Winkert]{Patrick Winkert}
\address[P. Winkert]{Technische Universit\"{a}t Berlin, Institut f\"{u}r Mathematik, Stra\ss{}e des 17.\,Juni 136, 10623 Berlin, Germany}
\email{winkert@math.tu-berlin.de}


\begin{document}

\begin{abstract}
	In this paper we study quasilinear elliptic Kirchhoff equations driven by a non-homogeneous operator with unbalanced growth and right-hand sides that consist of sub-linear, possibly singular, and super-linear reaction terms. Under very general assumptions we prove the existence of at least two solutions for such problems by using the fibering method along with an appropriate splitting of the associated Nehari manifold. In contrast to other works our treatment is very general, with much easier and shorter proofs as it was done in the literature before. Furthermore, the results presented in this paper cover a large class of second-order differential operators like the $p$-Laplacian, the $(p,q)$-Laplacian, the double phase operator, and the logarithmic double phase operator.
\end{abstract}

\maketitle

%********************************************************************
\section{Introduction}
%********************************************************************

Given a bounded domain $\Omega\subseteq\R^N$, $N\geq 2$, with Lipschitz boundary $\partial\Omega$, this paper deals with general Kirchhoff problems involving singular and super-linear reaction terms of the form
\begin{equation}\label{prob}\tag{${\rm P}_\lambda$}
	\left\{
		\begin{aligned}
			-m\left(\int_\Omega \H(x,|\nabla u|)\dx\right) \L(u) &= \lambda f(u) + g(u) \quad &&\text{in } \Omega, \\
			u &>0 \quad &&\text{in } \Omega, \\[1ex]
			u &=0 \quad &&\text{on } \partial\Omega,
		\end{aligned}
	\right.
\end{equation}
where $\lambda>0$ is a parameter, $f,g\colon (0,+\infty)\to(0,+\infty)$ are continuously differentiable functions, $m\colon [0,+\infty)\to[0,+\infty)$ is the so-called Kirchhoff function, $\H\colon \Omega\times[0,+\infty)\to[0,+\infty)$ is a generalized $N$-function, while $\L\colon W^{1,\H}_0(\Omega)\to W^{1,\H}_0(\Omega)^*$ is an operator (possibly non-homogeneous and with unbalanced growth) satisfying certain structure conditions. To be more precise, we suppose the following assumptions on the data of problem \eqref{prob}:
\begin{enumerate}[label=\textnormal{(H)},ref=\textnormal{H}]
	\item\label{H}
\begin{enumerate}[label=\textnormal{(H$_{m}$)},ref=\textnormal{H$_{m}$},leftmargin=1cm]
	\item \label{hypm}
		The function $m\colon [0,+\infty)\to[0,+\infty)$ is continuously differentiable, non-decreasing, and $m(s)>0$ for all $s>0$. In particular,
		\begin{align*}
			\eta:= \sup_{s>0} \frac{sm'(s)}{m(s)} \geq 0.
		\end{align*}
\end{enumerate}

\begin{enumerate}[label=\textnormal{(H$_{\L}$)},ref=\textnormal{H$_{\L}$},leftmargin=1cm]
	\item \label{hypL}
		The function $\H\colon \Omega\times[0,+\infty)\to[0,+\infty)$ is a generalized $N$-function such that $\H(x,\cdot)\in C^2(0,+\infty)$ for a.a.\,$x\in\Omega$ and
		\begin{enumerate}[label=\textnormal{(\roman*)},ref=\textnormal{\roman*},leftmargin=1cm,topsep=0.2cm,itemsep=0.2cm]
			\item \label{hypL1}
				$\displaystyle p:=\inf_{(x,s)\in\Omega\times(0,+\infty)}  \frac{s\partial_s \H(x,s)}{\H(x,s)}>1$;
			\item %\label{hypL2}
				$\displaystyle q:=\sup_{(x,s)\in\Omega\times(0,+\infty)}  \frac{s\partial_s \H(x,s)}{\H(x,s)}<p^*$;
			\item \label{hypL3}
				$\displaystyle l_-:=\inf_{(x,s)\in\Omega\times(0,+\infty)}  \frac{s\partial^2_{ss} \H(x,s)}{\partial_s\H(x,s)}>0$;
			\item %\label{hypL4}
				$\displaystyle l_+:=\sup_{(x,s)\in\Omega\times(0,+\infty)} \frac{s\partial^2_{ss} \H(x,s)}{\partial_s\H(x,s)}<+\infty$,
		\end{enumerate}
		with the Sobolev conjugate $p^*$ of $p$. Moreover, $\L\colon  W^{1,\H}_0(\Omega)\to W^{1,\H}_0(\Omega)^*$ is defined as
		\begin{align}\label{operator}
			\L(u) := \operatorname{div}\left(\partial_s H(x,|\nabla u|) \frac{\nabla u}{|\nabla u|}\right).
		\end{align}
		In addition, we suppose that
		\begin{align}\label{compemb}
			W^{1,\H}(\Omega)\hookrightarrow L^\H(\Omega) \quad \text{compactly}.
		\end{align}
\end{enumerate}

\begin{enumerate}[label=\textnormal{(H$_{f}$)},ref=\textnormal{H$_{f}$},leftmargin=1cm]
	\item \label{hypf}
		The function $f\colon (0,+\infty)\to(0,+\infty)$ is continuously differentiable and satisfies
		\begin{align*}
			&\liminf_{s\to 0^+} f(s) \in (0,+\infty],\\
			&\gamma_- := -\sup_{s>0} \frac{sf'(s)}{f(s)} > 1-p,\\
			&\gamma_+ := -\inf_{s>0} \frac{sf'(s)}{f(s)} < 1.
		\end{align*}
\end{enumerate}

\begin{enumerate}[label=\textnormal{(H$_{g}$)},ref=\textnormal{H$_{g}$},leftmargin=1cm]
	\item \label{hypg}
		The function $g\colon (0,+\infty)\to(0,+\infty)$ is continuously differentiable and satisfies
		\begin{align*}
			&r_- := 1 + \inf_{s>0} \frac{sg'(s)}{g(s)} > 1,\\
			&r_+ := 1 + \sup_{s>0} \frac{sg'(s)}{g(s)} < p^*.
		\end{align*}
\end{enumerate}

\begin{enumerate}[label=\textnormal{(H$_{C}$)},ref=\textnormal{H$_{C}$},leftmargin=1cm]
	\item \label{hypind}
		The following condition holds true:
		\begin{align*}%\tag{C}
			q\eta + l_+ < r_--1,
		\end{align*}
		where $q,l_+$ are defined in \eqref{hypL}, $\eta$ is given in \eqref{hypm}, and $r_-$ comes from \eqref{hypg}.
\end{enumerate}
\end{enumerate}

The following conclusions can be made from hypotheses \eqref{H}:
\begin{itemize}
	\item
		the condition $0<l_-\leq l_+<+\infty$ in \eqref{hypL} makes $\L$ a uniformly elliptic operator;
	\item
		\eqref{hypf} ensures that $f$ is sub-linear, possibly singular;
	\item
		\eqref{hypg} guarantees that $g$ is sub-critical;
	\item
		\eqref{hypind} is a super-linearity condition on $g$.
\end{itemize}

First, we mention that hypotheses \eqref{H} include the standard Kirchhoff function $m(s)=a+bs^\eta$ with $a,b\in\R^2\setminus\{(0,0)\}$, that means we allow degenerate Kirchhoff problems which create the most interesting models in applications. The following operators are included in hypotheses \eqref{H}, whereby we suppose in all cases that $1<p<N$, $p<q$, and $0 \leq \mu(\cdot) \in L^\infty(\Omega)$, while we assume $q<p^*:=\frac{Np}{N-p}$ for (i) and $q+\kappa<p^*$ for (ii)--(iii), where $\kappa:=\frac{e}{e+t_0}\in(0,1)$ and $t_0>0$ is the unique solution of $t=e\log(e+t)$:
\begin{enumerate}[leftmargin=0.7cm]
	\item[\textnormal{(i)}]
		Double phase operator:
		\begin{align*}%\label{double-phase-operator}
			\operatorname{div}\left(|\nabla u|^{p-2}\nabla u+\mu(x)|\nabla u|^{q-2}\nabla u\right)
		\end{align*}
		generated by the generalized $N$-function
		\begin{align}\label{N-function-double-phase}
			\H(x,s)=s^p+\mu(x)s^q  \quad\text{for all }(x,s)\in\Omega\times [0,\infty),
		\end{align}
		see Crespo-Blanco--Gasi\'{n}ski--Harjulehto--Winkert \cite{Crespo-Blanco-Gasinski-Harjulehto-Winkert-2022};
	\item[\textnormal{(ii)}]
		Logarithmic double phase operator:
		\begin{align}\label{logarithmic-double-phase-operator}
			\operatorname{div}\left( |\nabla u|^{p-2} \nabla u + \mu(x) \left[\log(e+|\nabla u|) + \frac{|\nabla u|}{q(e+|\nabla u|)}\right]|\nabla u|^{q-2} \nabla u \right)
		\end{align}
		generated by the generalized $N$-function
		\begin{align}\label{N-function-logarithmic-double-phase}
			\H(x,s)=s^p+\mu(x)s^q \log(e+s) \quad\text{for all }(x,s)\in\Omega\times [0,\infty),
		\end{align}
		where $e$ is the Euler number, see Arora--Crespo-Blanco--Winkert \cite{Arora-Crespo-Blanco-Winkert-2023};
	\item[\textnormal{(iii)}]
		Double phase operator with logarithmic perturbation:
		\begin{align}\label{logarithmic-double-phase-operator-2}
			\operatorname{div}\left(\left( |\nabla u|^{p-2} \nabla u + \mu(x)|\nabla u|^{q-2} \nabla u\right) \left[\log(e+|\nabla u|) + \frac{|\nabla u|}{q(e+|\nabla u|)}\right] \right),
		\end{align}
		generated by the generalized $N$-function
		\begin{align}\label{N-function-logarithmic-double-phase-2}
			\H(x,s) =[s^{p}+\mu(x)s^{q}]\log(e+s)\quad\text{for all }(x,s)\in\Omega\times [0,\infty),
		\end{align}
		where $e$ is the Euler number, see Lu--Vetro--Zeng \cite{Lu-Vetro-Zeng-2024}.
\end{enumerate}
We point out that, in the examples above, we do not need that $0\leq\mu(\cdot) \in C^{0,1}(\overline{\Omega})$ and
\begin{align}\label{density-condition2}
	\frac{q}{p}<1+\frac{1}{N},
\end{align}
as required quite often in the double phase setting. Note that \eqref{compemb} holds for \eqref{N-function-double-phase}, \eqref{N-function-logarithmic-double-phase}, and \eqref{N-function-logarithmic-double-phase-2}, see \cite[Proposition 2.18]{Crespo-Blanco-Gasinski-Harjulehto-Winkert-2022},  \cite[Proposition 3.9]{Arora-Crespo-Blanco-Winkert-2023}, and \cite[Proposition 2.24]{Lu-Vetro-Zeng-2024}, respectively, without supposing \eqref{density-condition2}. Sufficient conditions for the compact embedding in \eqref{compemb} to be true can be found in the book by Harjulehto--H\"{a}st\"{o} \cite[see Chapter 6.3]{Harjulehto-Hasto-2019} or the recent paper by Cianchi--Diening \cite[Theorem 3.7]{Cianchi-Diening-2024}. Concerning the nonlinearities on the right-hand side of \eqref{prob}, the choices  $f(s)=s^{-\gamma}$ and $g(s)=s^{r-1}$ are allowed for $0<\gamma<1<q<r<p^*$.

Our main result is the following theorem.

\begin{theorem}\label{mainthm}
	Let \eqref{H} be satisfied. Then there exists $\Lambda>0$ such that, for any $\lambda\in(0,\Lambda)$, problem \eqref{prob} admits two weak solutions with opposite energy sign.
\end{theorem}

The proof of Theorem \ref{mainthm} is based on the fibering method along with the corresponding Nehari manifold related to problem \eqref{prob}. Indeed, even though the energy functional $J\colon W^{1,\mathcal{H}}_0(\Omega)\to\R$ associated with \eqref{prob} is not $C^1$ (due to the presence of the singular term $f$), one can define the Nehari manifold to \eqref{prob} as
\begin{align*}
	\Ne = \{u\in W^{1,\H}_0(\Omega)\setminus\{0\}\colon \psi_u'(1)=0\},
\end{align*}
where $\psi_u\colon (0,+\infty)\to\R$ is the fibering map defined for any $u\in W^{1,\H}_0(\Omega)\setminus\{0\}$ by
\begin{align*}
	\psi_u(t):=J(tu)\quad \text{for all }t>0.
\end{align*}
The idea is then to split the Nehari manifold into three disjoint parts minimizing $J$ over two of them to get the required solutions with different energy sign. This method is not new, but it is the first time that it is applied to a very general setting and so no concrete, long calculations are needed. Indeed, we do not only cover the results obtained by Papageorgiou--Repov\v{s}--Vetro \cite{Papageorgiou-Repovs-Vetro-2021} ($(q,p)$-Laplacian), Papageorgiou--Win\-kert \cite{Papageorgiou-Winkert-2021} (weighted $p$-Laplacian), Liu--Dai--Papageorgiou--Winkert \cite{Liu-Dai-Papageorgiou-Winkert-2022} (double phase operator) or Arora--Fiscella--Mukherjee--Winkert \cite{Arora-Fiscella-Mukherjee-Winkert-2023} (Kirchhoff double phase operator), but we also have much easier and shorter proofs as in those papers and we also cover new operators within our setting, like the logarithmic double phase operators given in \eqref{logarithmic-double-phase-operator} and \eqref{logarithmic-double-phase-operator-2}.

In general, the use of the fibering method along with the Nehari manifold is a very powerful tool and has been further developed by the works of Dr\'{a}bek--Pohozaev \cite{Drabek-Pohozaev-1997} and Sun--Wu--Long \cite{Sun-Wu-Long-2001}. Subsequently, several authors have applied this method to various problems of singular type and non-singular type. We refer to works by Alves--Santos--Silva \cite{Alves-Santos-Silva-2022} (singular-superlinear Schr\"{o}dinger equations with indefinite-sign potential), Arora--Fiscella--Mukherjee--Winkert \cite{Arora-Fiscella-Mukherjee-Winkert-2022} (critical double phase Kirchhoff problems with singular
nonlinearity), Chen--Kuo--Wu \cite{Chen-Kuo-Wu-2011} (Kirchhoff Laplace equations), Fiscella--Mishra \cite{Fiscella-Mishra-2019} (fractional singular Kirchhoff problems), Kumar--R\u{a}dulescu--Sreenadh \cite{Kumar-Radulescu-Sreenadh-2020} (singular problems with unbalanced growth and critical exponent), Liu--Winkert \cite{Liu-Winkert-2022} (double phase problems in $\mathbb{R}^N$), Mukherjee--Sreenadh \cite{Mukherjee-Sreenadh-2019} (fractional $p$-Laplace problems), Tang--Cheng \cite{Tang-Chen-2017} (ground state solutions of Nehari--Pohozaev type for Kirchhoff-type problems with general potentials), Wang--Zhao--Zhao \cite{Wang-Zhao-Zhao-2013} (critical Laplace equations with singular term), see also the references therein. For a survey concerning singular problems, we address the reader to the overview article by Guarnotta--Livrea--Marano \cite{Guarnotta-Livrea-Marano-2022}. It should be mentioned that, in contrast to the results available in the literature (see the list above and also Candito--Guarnotta--Perera \cite{Candito-Guarnotta-Perera-2020} and Candito--Guarnotta--Livrea \cite{Candito-Guarnotta-Livrea-2022}), our method does not require the use of Hardy-Sobolev's inequality.

The paper is organized as follows. In Section \ref{sec2} we introduce our function space and recall some basic facts about generalized $N$-functions and related Musielak--Orlicz Sobolev spaces. Further, we prove some auxiliary results and give the precise definition of the Nehari manifold to problem \eqref{prob} including its splitting into three disjoint parts. Section \ref{sec3} discusses some basic estimates which are needed in the sequel while Section \ref{sec4} gives a detailed study of the Nehari manifold and its properties. Finally, in Section \ref{sec5}, we are able to prove Theorem \ref{mainthm}.

%********************************************************************
\section{Preliminaries}\label{sec2}
%********************************************************************
In this section we recall some basic definitions about $N$-functions, Musielak--Orlicz Sobolev spaces and its properties. These results are mainly taken from the monographs by Chlebicka--Gwiazda--\'{S}wierczewska-Gwiazda--Wr\'{o}blewska-Kami\'{n}ska \cite{Chlebicka-Gwiazda-Swierczewska-Gwiazda-Wroblewska-Kaminska-2021},  Diening--Harjulehto--H\"{a}st\"{o}--R$\mathring{\text{u}}$\v{z}i\v{c}ka \cite{Diening-Harjulehto-Hasto-Ruzicka-2011}, Harjulehto--H\"{a}st\"{o} \cite{Harjulehto-Hasto-2019}, Musielak \cite{Musielak-1983}, and Papageorgiou--Winkert \cite{Papageorgiou-Winkert-2024}.  We start with some definitions.

\begin{definition}%\label{def_phi-function}
	$~$
	\begin{enumerate}
		\item[\textnormal{(i)}]
			A continuous and convex function $\varphi\colon[0,\infty)\to[0,\infty)$ is said to be a $\Phi$-function if $\varphi(0)=0$ and $\varphi(t)>0$ for all $t >0$.
		\item[\textnormal{(ii)}]
			A function $\varphi\colon\Omega \times [0,\infty)\to[0,\infty)$ is said to be a generalized $\Phi$-function if $\varphi(\cdot,t)$ is measurable for all $t\geq 0$ and $\varphi(x,\cdot)$ is a $\Phi$-function for a.a.\,$x\in\Omega$. We denote the set of all generalized $\Phi$-functions on $\Omega$ by $\Phi(\Omega)$.
		\item[\textnormal{(iii)}]
			A function $\varphi\in\Phi(\Omega)$ is locally integrable if $\varphi(\cdot,t) \in L^{1}(\Omega)$ for all $t>0$.
		\item[\textnormal{(iv)}]
			A function $\varphi\in\Phi(\Omega)$ satisfies the $\Delta_2$-condition if there exist a positive constant $C$ and a nonnegative function $h\in L^1(\Omega)$ such that
			\begin{align*}
				\varphi(x,2t) \leq C\varphi(x,t)+h(x)
			\end{align*}
			for a.a.\,$x\in\Omega$ and for all $t\in [0,\infty)$.
		\item[\textnormal{(v)}]
			Given $\varphi, \psi \in \Phi(\Omega)$, we say that $\varphi$ is weaker than $\psi$, denoted by $\varphi \prec \psi$, if there exist two positive constants $C_1, C_2$ and a nonnegative function $h\in L^1(\Omega)$ such that
			\begin{align*}
				\varphi(x,t) \leq C_1 \psi(x,C_2t)+h(x)
			\end{align*}
			for a.a.\,$x\in \Omega$ and for all $t \in[0,\infty)$.
	\end{enumerate}
\end{definition}

For $\varphi \in \Phi(\Omega)$ we denote by $\rho_\varphi$ the corresponding modular given by
\begin{align*}%\label{modular_orlicz}
	\rho_\varphi(u):= \int_\Omega \varphi\left(x,|u|\right)\dx.
\end{align*}
Let $M(\Omega)$ be the set of all measurable functions $u\colon \Omega\to\R$. Then, the Musielak--Orlicz space $L^\varphi(\Omega)$ is defined by
\begin{align*}
	L^\varphi(\Omega):=\left \{u \in M(\Omega)\colon \text{there exists }\alpha>0 \text{ such that }\rho_\varphi(\alpha u)< +\infty \right \}
\end{align*}
equipped with the norm
\begin{align*}
	\|u\|_{\varphi}:=\inf \left\{\alpha >0 \colon \rho_\varphi \left(\frac{u}{\alpha}\right)\leq 1\right\}.
\end{align*}

The next proposition is taken from Musielak \cite[Theorem 7.7 and Theorem 8.5]{Musielak-1983}.

\begin{proposition}\label{prop_complete}
	$~$
	\begin{enumerate}
		\item[\textnormal{(i)}]
			Let $\varphi \in \Phi(\Omega)$. Then $\left(L^\varphi(\Omega),\|\cdot\|_\varphi\right)$ is a Banach space.
		\item[\textnormal{(ii)}]
			Let $\varphi,\psi \in \Phi(\Omega)$ be locally integrable with $\varphi \prec \psi$. Then
			\begin{align*}
				L^\psi(\Omega) \hookrightarrow L^\varphi(\Omega).
			\end{align*}
	\end{enumerate}
\end{proposition}

The following proposition can be found in the books by Musielak \cite[Theorem 8.13]{Musielak-1983} and Diening--Harjulehto--H\"{a}st\"{o}-R$\mathring{\text{u}}$\v{z}i\v{c}ka  \cite[Lemma 2.1.14]{Diening-Harjulehto-Hasto-Ruzicka-2011}.

\begin{proposition}\label{prop_delta_two_and_modular}
	Let $\varphi \in \Phi(\Omega)$.
	\begin{enumerate}
		\item[\textnormal{(i)}]
			If $\varphi$ satisfy the  $\Delta_2$-condition, then
			\begin{align*}
				L^\varphi(\Omega)=\left \{u \in M(\Omega)\colon \rho_\varphi(u)< +\infty \right \}.
			\end{align*}
		\item[\textnormal{(ii)}]
			Furthermore, if $u \in L^\varphi(\Omega)$, then $\rho_\varphi(u)<1$ (resp.\,$=1$; $>1)$ if and only if $\|u\|_\varphi<1$ (resp.\,$=1$; $>1$).
	\end{enumerate}
\end{proposition}

Now we can state the definition of a $N$-function.

\begin{definition}
	A function $\varphi\colon [0,\infty) \to [0,\infty)$ is called $N$-function if it is a $\Phi$-function such that
	\begin{align*}
		\lim_{t\to 0^+} \frac{\varphi(t)}{t}=0
		\quad\text{and}\quad
		\lim_{t\to\infty} \frac{\varphi(t)}{t}=\infty.
	\end{align*}
	We call a function $\varphi\colon\Omega\times [0,\infty)\to[0,\infty)$ a generalized
	$N$-function if $\varphi(\cdot,t)$ is measurable for all $t \in [0,\infty)$ and $\varphi(x,\cdot)$ is a $N$-function for a.a.\,$x\in\Omega$. We denote the class
	of all generalized $N$-functions by $N(\Omega)$. %Note that $\varphi^* \in N(\Omega)$ whenever $\varphi \in N(\Omega)$.
\end{definition}

Now, let $\varphi\in\Phi(\Omega)$. The corresponding Sobolev space $W^{1,\varphi}(\Omega)$ is defined by
\begin{align*}
	W^{1,\varphi}(\Omega) := \left \{u \in L^\varphi(\Omega) \colon |\nabla u| \in L^\varphi(\Omega) \right\}
\end{align*}
equipped with the norm
\begin{align*}
	\|u\|_{1,\varphi} = \|u\|_\varphi+\|\nabla u\|_\varphi
\end{align*}
where $\|\nabla u\|_\varphi=\| \, |\nabla u| \,\|_\varphi$. If $\varphi \in N(\Omega)$ is locally integrable, we denote by $W^{1,\varphi}_0(\Omega)$ the completion of $C^\infty_0(\Omega)$ in $W^{1,\varphi}(\Omega)$.

The next theorem gives a criterion when the Sobolev spaces are Banach spaces and also reflexive. This result can be found in Musielak \cite[Theorem 10.2]{Musielak-1983} and  Fan \cite[Proposition 1.7 and 1.8]{Fan-2012}.

\begin{theorem}\label{prop_complete2}
	Let $\varphi\in N(\Omega)$ be locally integrable such that
	\begin{align*}%\label{condition_ph}
		\inf_{x\in\Omega} \varphi(x,1)>0.
	\end{align*}
	Then the spaces $W^{1,\varphi}(\Omega)$ and $W^{1,\varphi}_0(\Omega)$ are separable Banach spaces. Moreover, they are reflexive if $L^\varphi(\Omega)$ is reflexive.
\end{theorem}

Let us now consider the generalized $N$-function $\mathcal{H}$ satisfying hypotheses \eqref{hypL}. First note, that from Lemma 2.3.16 in Chlebicka--Gwiazda--\'{S}wierczewska-Gwiazda--Wr\'{o}blewska-Kami\'{n}ska \cite{Chlebicka-Gwiazda-Swierczewska-Gwiazda-Wroblewska-Kaminska-2021}, we know that $\mathcal{H}$ satisfies the $\Delta_2$-condition and so, by Proposition \ref{prop_delta_two_and_modular}, the space $L^\mathcal{H}(\Omega)$ can be given by
\begin{align*}
	L^\mathcal{H}(\Omega)=\left \{u \in M(\Omega)\colon \rho_\mathcal{H}(u)< +\infty \right \}
\end{align*}
with the associated modular $\rho_\mathcal{H}(\cdot)$. Also, Corollary 3.5.5 in \cite{Chlebicka-Gwiazda-Swierczewska-Gwiazda-Wroblewska-Kaminska-2021} guarantees that $L^\mathcal{H}(\Omega)$ is reflexive and so, by Theorem \ref{prop_complete2}, the spaces $W^{1,\mathcal{H}}(\Omega)$ and $W^{1,\mathcal{H}}_0(\Omega)$ are separable and reflexive. Note that \eqref{compemb} implies the validity of the Poincar\'{e} inequality, i.e.,
\begin{align}\label{poincare-inequality}
	\|u\|_\mathcal{H} \leq C \|\nabla u\|_\mathcal{H} \quad\text{for all }u \in W^{1,\mathcal{H}}_0(\Omega).
\end{align}
We refer to the proof of Proposition 2.18 in \cite{Crespo-Blanco-Gasinski-Harjulehto-Winkert-2022} which can be done for any generalized $N$-function in the same way. Using \eqref{poincare-inequality}, we can equip the space $W^{1,\mathcal{H}}_0(\Omega)$ with the equivalent norm
\begin{align*}
	\|u\|=\|\nabla u\|_\mathcal{H}\quad\text{for all }u \in W^{1,\mathcal{H}}_0(\Omega).
\end{align*}
Note that the requirement to suppose \eqref{compemb} is very general.  Indeed, in Harjulehto--H\"{a}st\"{o} \cite[see Chapter 6.3]{Harjulehto-Hasto-2019} or Cianchi--Diening \cite[Theorem 3.7]{Cianchi-Diening-2024}  one can find sufficient conditions for \eqref{compemb} to hold and one key assumption is condition (A1), which says the following:
\begin{enumerate}
%	\item[\textnormal{(i)}]
%		A function $\varphi\in N(\Omega)$ satisfies \textnormal{(A0)} if there exists a constant $\beta\in (0,1]$ such that $\beta \leq \varphi^{-1}(x,1) \leq \frac{1}{\beta}$ for a.\,a.\,$x \in \Omega$.
	\item[$\bullet$]%[\textnormal{(ii)}]
		A generalized $N$-function $\varphi\colon\Omega \times [0,\infty)\to[0,\infty)$ satisfies \textnormal{(A1)} if there exists $\beta \in (0,1)$ such that
		\begin{align*}
			\beta \varphi^{-1}(x,t) \leq \varphi^{-1}(y,t)
		\end{align*}
	for every $t \in [1,\frac{1}{|B|}]$, for a.\,a.\,$x,y\in B\cap \Omega$
	and for every ball $B$ with $|B|\leq 1$.
\end{enumerate}
We avoided to suppose conditions like (A1) because the embedding \eqref{compemb} is more general than assumption (A1). In fact, in \cite{Arora-Crespo-Blanco-Winkert-2023}, \cite{Crespo-Blanco-Gasinski-Harjulehto-Winkert-2022}, and \cite{Lu-Vetro-Zeng-2024} the validity of \eqref{compemb} for the logarithmic double phase operator, the double phase operator, and the double phase operator with logarithmic perturbation have been proved without condition (A1). For (A1) to be true for these operators we have to require that $0\leq\mu(\cdot) \in C^{0,1}(\overline{\Omega})$ and
\begin{align}\label{density-condition}
	\frac{q}{p}<1+\frac{1}{N},
\end{align}
see \cite[Theorem 3.12]{Arora-Crespo-Blanco-Winkert-2023}, \cite[Theorem 2.23]{Crespo-Blanco-Gasinski-Harjulehto-Winkert-2022}, and \cite[Proposition 2.27]{Lu-Vetro-Zeng-2024}. However, the compactness of \eqref{compemb} still holds when $0\leq\mu(\cdot) \in L^\infty(\Omega)$ without supposing \eqref{density-condition}, see \cite[Proposition 3.9]{Arora-Crespo-Blanco-Winkert-2023}, \cite[Proposition 2.18]{Crespo-Blanco-Gasinski-Harjulehto-Winkert-2022}, and \cite[Proposition 2.24]{Lu-Vetro-Zeng-2024}.

Next, we introduce the following functions, useful to compare generalized $N$-functions with suitable power functions. To this end, for given $-\infty<\alpha\leq\beta<+\infty$, we define
\begin{align}\label{W}
	\underline{W}_\alpha^\beta(t) := \min\{t^\alpha,t^\beta\} \quad \text{and} \quad \overline{W}_\alpha^\beta(t) := \max\{t^\alpha,t^\beta\}.
\end{align}

The next proposition summarizes the information carried by the so-called  `indices', i.e., the quantities appearing in \eqref{indexk} below. Although the result is well-known for $N$-functions, for the sake of completeness we will sketch its proof in a more general case, where no convexity of functions is required.

\begin{proposition}\label{indexinfo}
	Let $K\colon [0,+\infty)\to[0,+\infty)$ be of class $C^2$, strictly increasing, and such that $K(0)=0$. Set $k:=K'$ and suppose
	\begin{equation}\label{mildsing}
		\lim_{s\to 0^+} sk(s)=0,
	\end{equation}
	as well as
	\begin{equation}\label{indexk}
		-\infty<i_k:= \inf_{s>0} \frac{sk'(s)}{k(s)}\leq\sup_{s>0} \frac{sk'(s)}{k(s)}=:s_k<+\infty.
	\end{equation}
	Then
	\begin{equation}\label{indexK}
		i_k+1\leq \inf_{s>0} \frac{sk(s)}{K(s)}\leq\sup_{s>0} \frac{sk(s)}{K(s)}\leq s_k+1.
	\end{equation}
	Moreover,
	\begin{equation}\label{growthk}
		k(1) \underline{W}_{i_k}^{s_k}(s)\leq k(s)\leq k(1)\overline{W}_{i_k}^{s_k}(s)
	\end{equation}
	and
	\begin{equation}\label{growthK}
		K(1)\underline{W}_{i_k+1}^{s_k+1}(s) \leq K(s)\leq K(1)\overline{W}_{i_k+1}^{s_k+1}(s)
	\end{equation}
	for all $s>0$.
\end{proposition}

\begin{proof}
	By \eqref{indexk} we have
	\begin{align}\label{proof-1}
		i_k k(t) \leq tk'(t) \leq s_k k(t) \quad \text{for all }t>0.
	\end{align}
	Integrating by parts, along with $K(0)=0$ and \eqref{mildsing}, yields
	\begin{align}\label{proof-2}
		i_k K(s) \leq sk(s)-K(s) \leq s_k K(s) \quad \text{for all }s>0,
	\end{align}
	ensuring \eqref{indexK}.

	Taking any $s\geq 1$ and integrating \eqref{proof-1} in $[1,s]$ we infer
	\begin{align*}
		i_k\int_1^s \frac{\dt}{t} \leq \int_1^s \frac{k'(t)}{k(t)} \dt \leq s_k \int_1^s \frac{\dt}{t}.
	\end{align*}
	Recalling that $k>0$ because of the monotonicity of $K$, we deduce
	\begin{align*}
		\log s^{i_k} \leq\log \frac{k(s)}{k(1)} \leq \log s^{s_k},
	\end{align*}
	which implies \eqref{growthk} for $s\geq 1$. Now suppose $s\in(0,1)$. Integrating \eqref{proof-1} in $[s,1]$ leads to
	\begin{align*}
		i_k \int_s^1 \frac{\dt}{t}  \leq \int_s^1 \frac{k'(t)}{k(t)} \dt \leq s_k \int_s^1 \frac{\dt}{t}.
	\end{align*}
	Thus,
	\begin{align}\label{nearzero}
		\log s^{-i_k} \leq \log \frac{k(1)}{k(s)} \leq \log s^{-s_k},
	\end{align}
	which gives \eqref{growthk} for $s\in(0,1)$. The proof of \eqref{growthK} is analogous, taking \eqref{proof-2} into account.
\end{proof}

\begin{remark}
	It is worth noticing that \eqref{mildsing} is automatically satisfied when $i_k>-1$, due to \eqref{nearzero}. This is the case of the $N$-function $K:=\H(x,\cdot)$ (since \eqref{hypL}\eqref{hypL3} forces $i_k>0$) and of the singular term $k:=f$ (see \eqref{hypf}).
\end{remark}

Adapting standard arguments for $N$-functions (see, e.g., Fukagai--Ito--Narukawa \cite[Lemma 2.1]{Fukagai-Ito-Narukawa-2006}), it is readily seen that the following result holds true.

\begin{proposition}\label{modularnorm}
	Let $\Phi$ be a generalized $N$-function of class $C^1$. Suppose that
	\begin{align*}
		a:=\inf_{(x,s)\in\Omega\times(0,+\infty)} \frac{s\partial_s \Phi(x,s)}{\Phi(x,s)}>1, \quad b:=\sup_{(x,s)\in\Omega\times(0,+\infty)} \frac{s\partial_s \Phi(x,s)}{\Phi(x,s)}<+\infty.
	\end{align*}
	Then
	\begin{align*}
		\underline{W}_a^b(\|u\|_\Phi) \leq \int_\Omega \Phi(x,|u|) \dx \leq \overline{W}_a^b(\|u\|_\Phi) \quad \text{for all } u \in L^\Phi(\Omega).
	\end{align*}
\end{proposition}

Next, we are going to prove the properties of the operator.

\begin{lemma}\label{S+}
    Let \eqref{hypL} be satisfied. Then $\L\colon  W^{1,\H}_0(\Omega)\to W^{1,\H}_0(\Omega)^*$ defined in \eqref{operator} is a strictly monotone operator and fulfills the ${\rm (S_+)}$-property.
\end{lemma}

\begin{proof}
	The result is a consequence of Proposition 3.12 by Crespo-Blanco \cite{Crespo-Blanco-2024}. The only nontrivial condition to verify is
	\begin{align*}
		\lim_{s\to+\infty} \partial^2_{ss} \H(x,c+s)(c-s)^2=+\infty \quad \text{for all } c>0.
	\end{align*}
	To this aim, it suffices to prove that
	\begin{align}\label{crespoblanco}
		\lim_{s\to+\infty} s^2\partial^2_{ss} \H(x,s)=+\infty,
	\end{align}
	since the change of variable $\tau=s+c$ yields
	\begin{align*}
		\lim_{s\to+\infty} \partial^2_{ss} \H(x,c+s)(c-s)^2=\lim_{\tau\to+\infty} \tau^2\partial^2_{ss}\H(x,\tau)\left(\frac{\tau-2c}{\tau}\right)^2 \quad \text{for all } c>0.
	\end{align*}
	Proposition \ref{indexinfo} (applied with $k=\partial_s\H(x,\cdot)$), \eqref{hypL}\eqref{hypL1}, \eqref{hypL}\eqref{hypL3}, and \eqref{growthK} entail
	\begin{align*}
	\lim_{s\to+\infty}s^2\partial^2_{ss}\H(x,s) &= \lim_{s\to+\infty}\frac{s\partial^2_{ss}\H(x,s)}{\partial_s\H(x,s)} s\partial_s\H(x,s) \\
    &\geq l_- \partial_s\H(x,1) \lim_{s\to+\infty}s^{l_-+1} = +\infty,
	\end{align*}
	ensuring \eqref{crespoblanco}.
\end{proof}

Next, we will make use of this simple real-analysis result.

\begin{proposition}\label{realanal}
	Let $\varphi\colon (0,+\infty)\to\R$ be a differentiable function such that
	\begin{enumerate}
		\item[\textnormal{(i)}]
			$\displaystyle{\limsup_{t\to 0^+} \varphi(t)<0}$ and $\displaystyle{\limsup_{t\to +\infty} \varphi(t)<0}$;
		\item[\textnormal{(ii)}]
			$\displaystyle{\max_{t\in(0,+\infty)} \varphi(t) > 0}$;
		\item[\textnormal{(iii)}]
			each zero of $\varphi$ is non-degenerate, i.e., $\varphi(t)=0$ implies $\varphi'(t)\neq 0$.
	\end{enumerate}
	Then there exist $0<t_1<t_2$ such that $\varphi(t_1)=\varphi(t_2)=0$ and $\varphi'(t_1)>0>\varphi'(t_2)$.
\end{proposition}

\begin{proof}
	Set
	\begin{align*}
		t_1=\inf\{t>0\colon \varphi(t)>0\}
		\quad\text{and}\quad
		t_2=\sup\{t>0\colon \varphi(t)>0\}.
	\end{align*}
	The sets are non-empty by hypothesis and since $\varphi$ is continuous, we have $t_1<t_2$. We will only reason for $t_1$, the argument for $t_2$ is analogous. By assumption we have $t_1>0$, and by continuity of $\varphi$ we infer $\varphi(t_1)\geq 0$. If  $\varphi(t_1)>0$, then there exists $\delta>0$ such that $\varphi(t_1-\delta)>0$, contradicting the minimality of $t_1$. Thus $\varphi(t_1)=0$ and so $\varphi'(t_1)\neq 0$ by the non-degeneracy hypothesis. Suppose by contradiction that $\varphi'(t_1)<0$. Then there exists $\delta>0$ such that $\varphi(t_1-\delta)>0$, again in contradiction with the minimality of $t_1$. Hence $\varphi'(t_1)>0$.
\end{proof}

In order to define an energy functional associated with \eqref{prob}, we consider the following odd extensions of $f$ and $g$:
\begin{align*}
	\tilde{f}(s):=
	\begin{cases}
		f(s) \quad &\text{if } s>0, \\
		0 \quad &\text{if } s=0, \\
		-f(-s) \quad &\text{if } s<0, \\
	\end{cases}
	\qquad
	\tilde{g}(s):=
	\begin{cases}
		g(s) \quad &\text{if } s>0, \\
		0 \quad &\text{if } s=0, \\
		-g(-s) \quad &\text{if } s<0. \\
	\end{cases}
\end{align*}
For simplification, we still call this extensions as $f$ and $g$, respectively. We also introduce the functions $F,G\colon \R\to\R$ defined as
\begin{align*}
	F(s) := \int_0^s f(t) \dt, \quad G(s) := \int_0^s g(t) \dt \quad \text{for all }s\in\R.
\end{align*}

We set
\begin{align*}
	M(s)=\int_0^s m(t)\dt \quad \text{for all } s\in\R
\end{align*}
and
\begin{align}\label{thetadef}
    \theta:=\sup_{s>0} \frac{sm(s)}{M(s)},
\end{align}
as well as
\begin{align*}
	\phi(\xi) := \int_\Omega \H(x,|\xi|) \dx \quad \text{for all }\xi\in L^\H(\Omega;\R^N).
\end{align*}

The energy functional $J\colon W^{1,\H}_0(\Omega)\to\R$ associated with \eqref{prob} is
\begin{align*}
	J(u) := M(\phi(\nabla u)) - \lambda \int_\Omega F(u) \dx - \int_\Omega G(u) \dx \quad \text{for all }u\in W^{1,\H}_0(\Omega).
\end{align*}
Due to the symmetries chosen in the construction of $F$ and $G$, we have $J(u)=J(|u|)$ for all $u\in W^{1,\H}_0(\Omega)$. Moreover, due to \eqref{H}, $J$ turns out to be weakly sequentially lower semi-continuous.

For any $u\in W^{1,\H}_0(\Omega)$ we define the fibering map $\psi_u\colon (0,+\infty)\to\R$ as
\begin{align*}
	\psi_u(t):=J(tu)=M(\phi(t\nabla u))-\lambda\int_\Omega F(tu) \dx - \int_\Omega G(tu) \dx \quad \text{for all }t>0.
\end{align*}
Note that $F$ is even and $F(0)=0$, so
\begin{align*}
	\int_\Omega F(tu) \dx = \int_{\Omega\cap \{u>0\}} F(tu) \dx + \int_{\Omega\cap \{u<0\}} F(-tu) \dx
\end{align*}
for all $t>0$ and $u\in W^{1,\H}_0(\Omega)$. Thus, exploiting \eqref{hypf}, Proposition \ref{indexinfo}, and Lebesgue's dominated converge theorem, besides recalling $f(0)=0$, one has
\begin{align*}
	\partial_t \left[\int_\Omega F(tu) \dx\right] = \int_\Omega f(tu)u \dx \quad \text{for all } t>0.
\end{align*}
Taking into account the fact that $s\mapsto f(s)s$ is even and vanishes at the origin, one can reason as above to obtain
\begin{align*}
	\partial^2_{tt} \left[\int_\Omega F(tu) \dx\right] = \int_\Omega f'(tu)u^2 \dx \quad \text{for all } t>0,
\end{align*}
where we set $f'(0):=0$. Analogous arguments hold for $G$.

Accordingly,
\begin{align*}
	\psi_u'(t) &= \langle J'(tu),u \rangle_{W^{1,\H}_0(\Omega)} = m(\phi(t\nabla u))\langle\phi'(t\nabla u),\nabla u\rangle_{L^\H(\Omega;\R^N)} \\
	&\quad -\lambda \int_\Omega f(tu)u \dx - \int_\Omega g(tu)u \dx
\end{align*}
and
\begin{align*}
	\psi_u''(t) &= m'(\phi(t\nabla u))\langle\phi'(t\nabla u),\nabla u\rangle_{L^\H(\Omega;\R^N)}^2 + m(\phi(t\nabla u)) \phi''(t\nabla u)(\nabla u,\nabla u) \\
	&\quad- \lambda \int_\Omega f'(tu)u^2 \dx - \int_\Omega g'(tu)u^2 \dx,
\end{align*}
where $\phi''(\xi)(\cdot,\cdot)$ represents the bilinear form on $L^\H(\Omega;\R^N)\times L^\H(\Omega;\R^N)$ induced by $\phi''(\xi)$. Owing to \eqref{hypL}, $\phi''(\xi)$ is positive definite for all $\xi\in L^\H(\Omega;\R^N)$. Notice that
\begin{equation}\label{psiscaling}
	\psi_u(t) = \psi_{tu}(1), \quad t\psi_u'(t)=\psi_{tu}'(1), \quad \text{and} \quad t^2 \psi''_u(t) = \psi''_{tu}(1)
\end{equation}
for all $t\in(0,+\infty)$.

\begin{remark}\label{differentiability}
	Using hypotheses \eqref{H}, Proposition \ref{indexinfo}, and Lebesgue's dominated converge theorem, it is readily seen that the maps $(t,u)\mapsto\psi_u(t)$, $(t,u)\mapsto\psi'_u(t)$, and $(t,u)\mapsto\psi''_u(t)$ are continuous in $(0,+\infty)\times W^{1,\H}_0(\Omega)$.
\end{remark}

The Nehari manifold $\Ne$ associated with $J$ is
\begin{align*}
	\Ne = \{u\in W^{1,\H}_0(\Omega)\setminus\{0\}\colon \psi_u'(1)=0\},
\end{align*}
which can be divided in the following sets:
\begin{align*}
	\Ne^+ &:= \{u\in \Ne\colon  \psi_u''(1)>0\}, \\
	\Ne^0 &:= \{u\in \Ne\colon \psi_u''(1)=0\}, \\
	\Ne^- &:= \{u\in \Ne\colon  \psi_u''(1)<0\}.
\end{align*}
Due to the symmetry of $J$, both $u\mapsto\psi'_u(1)$ and $u\mapsto\psi''_u(1)$ are even. Thus, if $u\in\Ne$ then $|u|\in\Ne$, and the same holds for $\Ne^+$, $\Ne^-$, and $\Ne^0$. For any $u\in W^{1,\H}_0(\Omega)\setminus\{0\}$ we define
\begin{align*}
	E^+_u &:= \{t\in(0,+\infty)\colon  tu\in\Ne^+\}, \\
	E^0_u &:= \{t\in(0,+\infty)\colon  tu\in\Ne^0\}, \\
	E^-_u &:= \{t\in(0,+\infty)\colon tu\in\Ne^-\}.
\end{align*}
We will say that $E_u^+<E_u^-$ if $t^+<t^-$ for all $t^\pm\in E^\pm_u$. Furthermore, for all $\xi\in L^\H(\Omega;\R^N)$, we set
\begin{align*}
	A(\xi) &:= m(\phi(\xi))\langle\phi'(\xi),\xi\rangle_{L^\H(\Omega;\R^N)}, \\
	B(\xi) &:= m'(\phi(\xi))\langle\phi'(\xi),\xi\rangle_{L^\H(\Omega;\R^N)}^2 + m(\phi(\xi)) \phi''(\xi)(\xi,\xi),
\end{align*}
which represent the principal parts of $\psi'_u(1)$ and $\psi''_u(1)$, respectively.

To simplify matters, we will omit the subscripts in the duality brackets when the context is clear. As usual, the generic constants $c,C>0$ may change their value at each place.

%********************************************************************
\section{Basic estimates}\label{sec3}
%********************************************************************
In this section we will discuss some basic estimates which are needed in the sequel. We start with the following lemma.

\begin{lemma}\label{basicest1}
	Suppose \eqref{hypm} and \eqref{hypL} to be satisfied. Then, for all $\xi\in L^\H(\Omega;\R^N)$, the following hold:
	\begin{enumerate}[label={${\rm (a_\arabic*)}$},ref={${\rm a_\arabic*}$},itemsep=5pt]
		\item\label{ord0-1}
			$pM(\phi(\xi)) \leq A(\xi) \leq q\theta M(\phi(\xi))$;
		\item\label{ord1-2}
			$l_-A(\xi) \leq B(\xi) \leq (q\eta+l_+)A(\xi)$;
		\item\label{normcomp}
			$M(1)\underline{W}_p^{q\theta}(\|\xi\|_\H) \leq M(\phi(\xi)) \leq M(1)\overline{W}_p^{q\theta}(\|\xi\|_\H)$,
	\end{enumerate}
	where $\underline{W}_p^{q\theta},\overline{W}_p^{q\theta}$ are as in \eqref{W}, while $\theta$ is defined in \eqref{thetadef}.
\end{lemma}

\begin{proof}
	We fix $\xi\in L^\H(\Omega;\R^N)$. Reasoning as in the proof of \eqref{indexK}, the monotonicity of $m$ gives
	\begin{equation}\label{Mind}
		\inf_{s>0} \frac{sm(s)}{M(s)} \geq 1.
	\end{equation}
	Proposition \ref{indexinfo}, \eqref{hypL}, and \eqref{thetadef} yield
	\begin{equation}\label{est1above}
		\begin{aligned}
			m(\phi(\xi))\langle\phi'(\xi),\xi\rangle
			&= m(\phi(\xi)) \int_\Omega \partial_s\H(x,|\xi|)|\xi| \dx\\
			&\leq q m(\phi(\xi)) \phi(\xi) \leq q\theta M(\phi(\xi)).
		\end{aligned}
	\end{equation}
	Hence, using \eqref{hypL} and \eqref{Mind},
	\begin{equation}\label{est1below}
		m(\phi(\xi))\langle\phi'(\xi),\xi\rangle \geq p m(\phi(\xi)) \phi(\xi) \geq pM(\phi(\xi)).
	\end{equation}
	Putting together \eqref{est1above} and \eqref{est1below} yields \eqref{ord0-1}.

	Reasoning as above, from \eqref{hypL} and \eqref{hypm} we infer
	\begin{equation}\label{est2part1}
		m'(\phi(\xi))\langle \phi'(\xi),\xi \rangle \leq qm'(\phi(\xi))\phi(\xi) \leq q\eta m(\phi(\xi)),
	\end{equation}
	and
	\begin{equation}\label{est2part2}
		\phi''(\xi)(\xi,\xi) = \int_\Omega \partial^2_{ss} \H(x,|\xi|)|\xi|^2 \dx \leq l_+ \int_\Omega \partial_s \H(x,|\xi|)|\xi| \dx = l_+ \langle \phi'(\xi),\xi \rangle.
	\end{equation}
	Summing \eqref{est2part1} multiplied by $\langle \phi'(\xi),\xi \rangle$ with \eqref{est2part2} multiplied by $m(\phi(\xi))$ yields
	\begin{align*}
		B(\xi) \leq (q\eta+l_+)A(\xi).
	\end{align*}
	On the other hand, arguing as in \eqref{est2part2},
	\begin{align*}
		B(\xi) \geq m(\phi(\xi))\phi''(\xi)(\xi,\xi) \geq l_- m(\phi(\xi)) \langle \phi'(\xi),\xi \rangle = l_- A(\xi),
	\end{align*}
	which concludes the proof of \eqref{ord1-2}.

	Let us show \eqref{normcomp}. According to Proposition \ref{modularnorm} and \eqref{hypL} we have
	\begin{equation}\label{est3part1}
		\underline{W}_p^q(\|\xi\|_\H)\leq \phi(\xi)\leq \overline{W}_p^q(\|\xi\|_\H),
	\end{equation}
	while Proposition \ref{indexinfo}, \eqref{hypm}, and \eqref{Mind} ensure
	\begin{equation}\label{est3part2}
		M(1)\underline{W}_1^\theta(s) \leq M(s) \leq M(1)\overline{W}_1^\theta(s) \quad \text{for all } s>0.
	\end{equation}
	Thus, \eqref{est3part1} and \eqref{est3part2} together lead to
	\begin{align*}
		M(\phi(\xi)) \leq M(1)\overline{W}_1^\theta(\phi(\xi)) \leq M(1)\overline{W}_1^\theta(\overline{W}_p^q(\|\xi\|_\H)) = M(1)\overline{W}_p^{q\theta}(\|\xi\|_\H)
	\end{align*}
	and
	\begin{align*}
		M(\phi(\xi)) \geq M(1)\underline{W}_1^\theta(\phi(\xi)) \geq M(1)\underline{W}_1^\theta(\underline{W}_p^q(\|\xi\|_\H)) = M(1)\underline{W}_p^{q\theta}(\|\xi\|_\H),
	\end{align*}
	which gives \eqref{normcomp}.
\end{proof}

Note that the estimates contained in Lemma \ref{basicest1} will allow us to have controls of type
\begin{align*}
	B(\xi) \simeq A(\xi) \simeq M(\phi(\xi)), \quad \underline{W}_p^{q\theta}(\|\xi\|_\H) \lesssim M(\phi(\xi)) \lesssim \overline{W}_p^{q\theta}(\|\xi\|_\H),
\end{align*}
with $\underline{W}_p^{q\theta},\overline{W}_p^{q\theta}$  as in \eqref{W}.

\begin{lemma}\label{basicest2}
	Under the hypotheses \eqref{H}, for all $u\in W^{1,\H}_0(\Omega)$ one has
	\begin{enumerate}[label={${\rm (b_\arabic*)}$},ref={${\rm b_\arabic*}$},itemsep=5pt]
		\item\label{Fest}
			$\displaystyle{\int_\Omega F(u) \dx \leq C\overline{W}_{1-\gamma_+}^{1-\gamma_-}(\|\nabla u\|_\H)}$,
		\item\label{Gest}
			$c\underline{W}_{r_-}^{r_+}(\|u\|_\H) \leq \displaystyle{\int_\Omega G(u) \dx \leq C\overline{W}_{r_-}^{r_+}(\|\nabla u\|_\H)}$,
	\end{enumerate}
	for some $c,C>0$ independent of $u$.
\end{lemma}

\begin{proof}
	Fix any $u\in W^{1,\H}_0(\Omega)$. Exploiting Proposition \ref{indexinfo} and \eqref{hypf} we have
	\begin{align*}
		i_F:=\inf_{s>0} \frac{sf(s)}{F(s)} \geq 1-\gamma_+>0, \quad s_F:=\sup_{s>0}\frac{sf(s)}{F(s)} \leq 1-\gamma_-<p.
	\end{align*}
	Thus, using Proposition \ref{indexinfo} again, besides H\"older's inequality and $W^{1,\H}_0(\Omega)\hookrightarrow L^\H(\Omega)\hookrightarrow L^p(\Omega)$ (see \eqref{hypL} and Proposition \ref{prop_complete}), we get
	\begin{align*}
		\int_\Omega F(u) \dx
		&\leq C\left(\int_{\{u\geq 1\}} u^{s_F} \dx + \int_{\{u<1\}} u^{i_F} \dx \right) \leq C(\|u\|_p^{s_F}+\|u\|_p^{i_F}) \\
		&\leq C(\|\nabla u\|_\H^{s_F}+\|\nabla u\|_\H^{i_F}) \leq C\overline{W}_{i_F}^{s_F}(\|\nabla u\|_\H) \leq C\overline{W}_{1-\gamma_+}^{1-\gamma_-}(\|\nabla u\|_\H),
	\end{align*}
	for a suitable $C>0$ changing at each place. Hence \eqref{Fest} is established.

	In order to show \eqref{Gest}, observe that Proposition \ref{indexinfo} and \eqref{hypg} entail
	\begin{align*}
		i_G:=\inf_{s>0}\frac{sg(s)}{G(s)} \geq r_->1, \quad s_G:=\sup_{s>0}\frac{sg(s)}{G(s)}\leq r_+<p^*.
	\end{align*}
	In particular, $G$ is a generalized $N$-function satisfying the hypotheses of Proposition \ref{modularnorm}. Moreover, owing to \eqref{Mind}, one has $\theta \geq 1$. Then \eqref{hypind} and Proposition \ref{indexinfo} yield
	\begin{equation}\label{superlinear}
		r_- > q\eta+l_++1\geq q(\eta+1)\geq q\theta \geq q \geq p.
	\end{equation}
	Since $p\leq q<r_-\leq r_+<p^*$, by Sobolev's embedding and Proposition \ref{prop_complete} one has $W^{1,\H}_0(\Omega) \hookrightarrow W^{1,p}_0(\Omega) \hookrightarrow L^{p^*}(\Omega) \hookrightarrow L^G(\Omega) \hookrightarrow L^\H(\Omega)$. Thus, Proposition \ref{modularnorm} ensures
	\begin{align*}
		\int_\Omega G(u) \dx \leq \overline{W}_{r_-}^{r_+}(\|u\|_G) \leq C\overline{W}_{r_-}^{r_+}(\|\nabla u\|_\H).
	\end{align*}
        Analogously,
        \begin{align*}
		\int_\Omega G(u) \dx \geq \underline{W}_{r_-}^{r_+}(\|u\|_G) \geq c\underline{W}_{r_-}^{r_+}(\|u\|_\H).
	\end{align*}
	establishing \eqref{Gest}.
\end{proof}

%********************************************************************
\section{Analysis of the Nehari manifold}\label{sec4}
%********************************************************************

In this section we study the Nehari manifold and its properties.

\begin{lemma}\label{coercivity}
	Let hypotheses \eqref{H} be satisfied. Then $J |_{\Ne}$ is coercive.
\end{lemma}

\begin{proof}
	Let $\{u_n\}_{n\in\N}\subseteq \Ne$ be such that $\|u_n\|\to\infty$. By definition of $\Ne$ we have
	\begin{equation}\label{nehariseq}
		\int_\Omega g(u_n)u_n \dx = m(\phi(\nabla u_n))\langle\phi'(\nabla u_n),\nabla u_n\rangle - \lambda \int_\Omega f(u_n)u_n \dx
	\end{equation}
	for all $n\in\N$. Thus Proposition \ref{indexinfo}, \eqref{hypg}, \eqref{nehariseq}, and Lemmas \ref{basicest1} as well as \ref{basicest2} imply that, whenever $\|\nabla u_n\|_\H$ is sufficiently large,
	\begin{align*}
		J(u_n)
		&= M(\phi(\nabla u_n)) - \lambda \int_\Omega F(u_n) \dx - \int_\Omega \left(\frac{g(u_n)u_n}{G(u_n)}\right)^{-1} g(u_n)u_n \dx \\
		&\geq M(\phi(\nabla u_n)) - \lambda \int_\Omega F(u_n) \dx - \frac{1}{r_-} \int_\Omega g(u_n)u_n \dx \\
		&= M(\phi(\nabla u_n))-\frac{1}{r_-}m(\phi(\nabla u_n))\langle\phi'(\nabla u_n),\nabla u_n\rangle\\
		&\quad - \lambda \left[ \int_\Omega F(u_n) \dx-\frac{1}{r_-}\int_\Omega f(u_n)u_n \dx \right] \\
		&\geq \left(1-\frac{q\theta}{r_-}\right)M(\phi(\nabla u_n)) - \lambda \int_\Omega F(u_n) \dx \geq c\|\nabla u_n\|_\H^p - \lambda C \|\nabla u_n\|_\H^{1-\gamma_-},
	\end{align*}
	with $c,C>0$, where we also used that $r_->q\theta$, owing to \eqref{superlinear}. Recalling $1-\gamma_-<p$ yields $J(u_n)\to+\infty$.
\end{proof}

\begin{lemma}\label{N0empty}
	Let hypotheses \eqref{H} be satisfied. Then there exist $D_1=D_1(\lambda)>0$ and $D_2>0$ such that
	\begin{equation}\label{neharibounds}
		\|u^+\| \leq D_1 \quad \text{and} \quad \|u^-\| \geq D_2
	\end{equation}
	for all $u^+\in\Ne^+\cup\Ne^0$ and for all $u^-\in\Ne^-\cup\Ne^0$. Moreover,
	\begin{equation}\label{smallD1}
		\lim_{\lambda\to 0^+} D_1(\lambda)=0,
	\end{equation}
	so there exists $\Lambda_1>0$ such that $D_1<D_2$ and $\Ne^0=\emptyset$ for any $\lambda\in(0,\Lambda_1)$.
\end{lemma}

\begin{proof}
	We take $u^+\in\Ne^+\cup\Ne^0$ and $u^-\in\Ne^-\cup\Ne^0$. Then, by the definitions of $\Ne$, $\Ne^\pm$, and $\Ne^0$, one has
	\begin{equation}\label{nehari}
		A(\nabla u^\pm) = \lambda\int_\Omega f(u^\pm)u^\pm \dx + \int_\Omega g(u^\pm)u^\pm \dx,
	\end{equation}
	as well as
	\begin{equation}\label{nehari+}
		B(\nabla u^+) \geq \lambda\int_\Omega f'(u^+)(u^+)^2 \dx + \int_\Omega g'(u^+)(u^+)^2 \dx
	\end{equation}
	and
	\begin{equation}\label{nehari-}
		B(\nabla u^-) \leq \lambda\int_\Omega f'(u^-)(u^-)^2 \dx + \int_\Omega g'(u^-)(u^-)^2 \dx.
	\end{equation}

	Let us reason for $u^+$. According to \eqref{hypg}, along with \eqref{nehari} and \eqref{nehari+}, we get
	\begin{align*}
		&(r_--1)A(\nabla u^+)-B(\nabla u^+) \\
		&\leq \lambda \int_\Omega f(u^+)u^+ \left(r_--1-\frac{f'(u^+)u^+}{f(u^+)}\right) \dx\\
		&\quad + \int_\Omega g(u^+)u^+ \left(r_--1-\frac{g'(u^+)u^+}{g(u^+)}\right) \dx \\
		&\leq \lambda \int_\Omega f(u^+)u^+ \left(r_--1-\frac{f'(u^+)u^+}{f(u^+)}\right) \dx.
	\end{align*}
	Thus, owing to Lemma \ref{basicest1}, \eqref{hypf}, \eqref{superlinear}, Proposition \ref{indexinfo}, and Lemma \ref{basicest2}, we have
	\begin{align*}
		&c\underline{W}_p^{q\theta}(\|\nabla u^+\|_\H) \leq (r_--1-q\eta-l_+)A(\nabla u^+) \leq (r_--1)A(\nabla u^+)-B(\nabla u^+) \\
		&\leq \lambda \int_\Omega f(u^+)u^+ \left(r_--1-\frac{f'(u^+)u^+}{f(u^+)}\right) \dx \leq \lambda(r_--1+\gamma_+)\int_\Omega f(u^+)u^+ \dx \\
		&\leq \lambda(r_--1+\gamma_+)(1-\gamma_-) \int_\Omega F(u^+) \dx \leq \lambda C\overline{W}_{1-\gamma_+}^{1-\gamma_-}(\|\nabla u^+\|_\H),
	\end{align*}
	for suitable $c,C>0$. Since $1-\gamma_-<p$, $\{\nabla u^+\}$ is bounded in $L^\H(\Omega)$, whence $\{u^+\}$ is bounded in $W^{1,\H}(\Omega)$. Moreover, $\|u^+\|\to 0$ as $\lambda \to 0$, ensuring \eqref{smallD1}.

	Now we focus on $u^-$, reasoning as above. Hypothesis \eqref{hypf}, besides \eqref{nehari} and \eqref{nehari-}, yields
	\begin{align*}
		&\gamma_- A(\nabla u^-)+B(\nabla u^-) \\
		&\leq \lambda \int_\Omega f(u^-)u^- \left(\gamma_-+\frac{f'(u^-)u^-}{f(u^-)}\right) \dx + \int_\Omega g(u^-)u^- \left(\gamma_-+\frac{g'(u^-)u^-}{g(u^-)}\right) \dx \\
		&\leq \int_\Omega g(u^-)u^- \left(\gamma_-+\frac{g'(u^-)u^-}{g(u^-)}\right) \dx.
	\end{align*}
	Exploiting Lemma \ref{basicest1}, \eqref{hypg}, \eqref{superlinear}, Proposition \ref{indexinfo}, and Lemma \ref{basicest2}, we deduce
	\begin{align*}
		c\underline{W}_p^{q\theta}(\|\nabla u^-\|_\H)
		&\leq \gamma_- A(\nabla u^-) \leq \gamma_- A(\nabla u^-)+B(\nabla u^-) \\
		&\leq \int_\Omega g(u^-)u^- \left(\gamma_-+\frac{g'(u^-)u^-}{g(u^-)}\right) \dx\\
		&\leq (r_+-1+\gamma_-) \int_\Omega g(u^-)u^- \dx \\
		&\leq (r_+-1+\gamma_-)r_+ \int_\Omega G(u^-) \dx \leq C\overline{W}_{r_-}^{r_+}(\|\nabla u^-\|_\H).
	\end{align*}
	Since $q\theta<r_-$ by \eqref{superlinear}, there exists a positive lower bound for $\|\nabla u^-\|_\H$, namely $\|u^-\|$. Hence, \eqref{neharibounds} is established.

	To conclude, observe that \eqref{smallD1} provides $\Lambda_1>0$ such that $D_1<D_2$ for all $\lambda\in(0,\Lambda_1)$, with $D_1,D_2$ as in \eqref{neharibounds}. Let $\lambda\in(0,\Lambda_1)$. If, by contradiction, there exists $u\in\Ne^0$, then applying \eqref{neharibounds} with $u^+=u^-=u$ entails $\|u\|<D_1<D_2<\|u\|$, which is a contradiction. Accordingly, $\Ne^0=\emptyset$ for all $\lambda\in(0,\Lambda_1)$.
\end{proof}

\begin{lemma}\label{fiberstruct}
	Let hypotheses \eqref{H} be satisfied. Then there exists $\Lambda_2>0$ such that for all $\lambda\in(0,\Lambda_2)$ the following statement holds true: for any $u\in W^{1,\H}_0(\Omega)\setminus\{0\}$ there exist unique $0<t_u^+<t_u^-$ such that $t_u^+u\in\Ne^+$ and $t_u^-u\in\Ne^-$.
\end{lemma}

\begin{proof}
	Fix any $u\in W^{1,\H}_0(\Omega)\setminus\{0\}$. We are going to apply Proposition \ref{realanal} to the function $\psi'_u$. To this end, using Lebesgue's dominated convergence theorem and Fatou's lemma, along with \eqref{hypm}, \eqref{hypL}, \eqref{hypg}, \eqref{hypf}, and the symmetry of $f$, ensures
	\begin{align*}
		\limsup_{t\to 0^+} \psi_u'(t)
		&= \lim_{t\to 0^+} \left[m(\phi(t\nabla u))\langle \phi'(t\nabla u),\nabla u\rangle -\int_\Omega g(tu)u \dx\right]\\
		&\quad - \lambda \liminf_{t\to 0^+} \int_\Omega f(tu)u \dx \\
		&= - \lambda \liminf_{t\to 0^+} \int_\Omega f(t|u|)|u| \dx \leq -\lambda \left(\liminf_{t\to 0^+} f(t)\right) \int_\Omega |u| \dx<0.
	\end{align*}
	Exploiting \eqref{hypg}, Proposition \ref{indexinfo}, and Lemmas \ref{basicest1} as well as \ref{basicest2} we get, for all $t$ sufficiently large,
	\begin{align*}
		A(t\nabla u)-\int_\Omega g(tu)tu \dx &\leq A(t\nabla u)-r_-\int_\Omega G(tu) \dx\\
		&\leq C(t\|\nabla u\|_\H)^{q\theta} - c(t\|u\|_\H)^{r-},
	\end{align*}
	for some $c,C>0$. Hence, recalling \eqref{psiscaling} and $q\theta<r_-$ (see \eqref{superlinear}),
	\begin{align*}
		\limsup_{t\to +\infty} \psi'_u(t) = \limsup_{t\to+\infty} \frac{1}{t}\psi'_{tu}(1)
		&\leq \limsup_{t\to+\infty} \frac{1}{t}\left[A(t\nabla u)-\int_\Omega g(tu)tu \dx \right] \\
		&\leq \limsup_{t\to+\infty} \left[Ct^{q\theta-1}\|\nabla u\|_\H^{q\theta} - ct^{r_--1}\|u\|_\H^{r-}\right] = -\infty.
	\end{align*}

	Reasoning as above, \eqref{psiscaling}, \eqref{hypf}, \eqref{hypg}, Proposition \ref{indexinfo}, and Lemmas \ref{basicest1}--\ref{basicest2} imply
	\begin{equation}\label{lowerbound}
		\begin{aligned}
			t\psi'_u(t)
			&= \psi'_{tu}(1)\\
			&\geq c\underline{W}_p^{q\theta}(t\|\nabla u\|_\H)-C\left[\lambda\overline{W}_{1-\gamma_+}^{1-\gamma_-}(t\|\nabla u\|_\H)+\overline{W}_{r_-}^{r_+}(t\|\nabla u\|_\H)\right]
		\end{aligned}
	\end{equation}
	for all $t>0$ and opportune $c,C>0$. In order to have $\max_{t>0} \psi'_u(t) > 0$, from \eqref{lowerbound} it suffices that
	\begin{align*}
		\lambda < \frac{c\underline{W}_p^{q\theta}(\overline{t}\|\nabla u\|_\H)-C\overline{W}_{r_-}^{r_+}(\overline{t}\|\nabla u\|_\H)}{C\overline{W}_{1-\gamma_+}^{1-\gamma_-}(\overline{t}\|\nabla u\|_\H)} \quad \text{for some } \overline{t}=\overline{t}(u)>0.
	\end{align*}
	To this aim we choose $\overline{t}:=\frac{\rho}{\|\nabla u\|_\H}$ with $\rho\in(0,1)$ such that
	\begin{align*}
		\hat{\Lambda} := \frac{c\rho^{q\theta}-C\rho^{r_-}}{C\rho^{1-\gamma_+}} > 0.
	\end{align*}
	This choice is possible since $q\theta<r_-$. In particular, $\max_{t>0} \psi'_u(t)>0$ whenever $\lambda<\hat{\Lambda}$. Moreover, note that $\overline{t}$ depends on $u$, but $\hat{\Lambda}$ does not.

	Set $\Lambda_2:=\min\{\Lambda_1,\hat{\Lambda}\}$. Assuming $\lambda<\Lambda_2$, Lemma \ref{N0empty} ensures $\Ne^0=\emptyset$, so that each zero of $\psi'_u$ is not a zero of $\psi''_u$. Indeed, if $t\in(0,+\infty)$ is a zero of $\psi'_u$, then \eqref{psiscaling} yields
	\begin{align*}
		\psi'_{tu}(1)=t\psi'_u(t)=0 \quad \text{and} \quad \psi''_u(t)=\frac{\psi''_{tu}(1)}{t^2}\neq 0.
	\end{align*}
	Hence Proposition \ref{realanal} provides $0<t_u^+<t_u^-$ such that $t_u^\pm u\in\Ne^\pm$.

	Next we show the uniqueness. The argument above ensures that the sets $E^\pm_u$ are non-empty. We show that $E_u^+<E_u^-$. By contradiction suppose that there exist $t^-\leq t^+$ such that $t^\pm\in E_u^\pm$. Then, by Lemma \ref{N0empty},
	\begin{align*}
		D_2 \leq t^- \|u\| \leq t^+ \|u\| \leq D_1,
	\end{align*}
	contradicting $D_1<D_2$, which holds for all $\lambda<\Lambda_1$.

	Now we prove that $E_u^+$ is a singleton. An analogous argument guarantees the same property for $E_u^-$. By contradiction, let $t_1^+,t_2^+\in E^+_u$ fulfill $t_1^+<t_2^+$. Then there exists $\delta\in(0,\frac{1}{2}(t_2^+-t_1^+))$ such that $\psi_u'(t_1^++\delta)>0>\psi_u'(t_2^+-\delta)$. By Bolzano's theorem there exists $t_0\in (t_1^+,t_2^+)$ such that $t_0 u\in \Ne$. Consider
	\begin{align*}
		\overline{t} = \sup \left\{t\in(t_1^++\delta,t_2^+-\delta)\colon  tu\in \Ne\right\}.
	\end{align*}
	By continuity of $\psi'_u$ we deduce that $\overline{t}u\in\Ne$, so $\overline{t}\in E_u^+ \cup E_u^- \cup E_u^0$. Since $\Ne^0=\emptyset$ one has $\overline{t}\notin E_u^0$. On the other hand, again by Bolzano's theorem, $\overline{t}\in E_u^+$ would contradict the maximality of $\overline{t}$. Hence $\overline{t}\in E_u^-$. Since $\overline{t}<t_2^+$, we get a contradiction with $E_u^+<E_u^-$. We deduce $E_u^+=\{t_1^+\}$.
\end{proof}

\begin{lemma}\label{energysign}
	Let hypotheses \eqref{H} be satisfied and $\Lambda_2$ be as in Lemma \ref{fiberstruct}. Then $J(u)<0$ for all $u\in\Ne^+$ provided $\lambda<\Lambda_2$. Moreover, there exists $\Lambda_3>0$ such that for all $\lambda\in(0,\Lambda_3)$ the following assertion is true: there exists $\sigma>0$ such that $J(v)\geq \sigma$ for all $v\in\Ne^-$.
\end{lemma}

\begin{proof}
	Suppose $\lambda<\Lambda_2$, where $\Lambda_2$ is from Lemma \ref{fiberstruct}. Pick any $u\in\Ne^+$. Owing to Lemma \ref{fiberstruct}, one has $\psi_u'(t)<0$ for all $t\in(0,1)$. Indeed, $\psi_u'(t)<0$ near $t=0$ and, if $\psi_u'(t)=0$ for some $t\in(0,1)$, then $t\notin E_u^+\cup E_u^-\cup E_u^0$, according to the fact that $E_u^+=\{1\}$, $E_u^+<E_u^-$, and $E_u^0=\emptyset$, respectively. Hence
	\begin{align*}
		J(u)=\psi_u(1)<\psi_u(0)=J(0)=0.
	\end{align*}

	Now consider an arbitrary $v\in\Ne^-$. Reasoning as in \eqref{lowerbound} we get
	\begin{align*}
		J(tv)= \psi_{tv}(1)\geq c\underline{W}_p^{q\theta}(t\|\nabla v\|_\H)-C\left[\lambda \overline{W}_{1-\gamma_+}^{1-\gamma_-}(t\|\nabla v\|_\H)+\overline{W}_{r_-}^{r_+}(t\|\nabla v\|_\H)\right]
	\end{align*}
	for all $t>0$ with some $c,C>0$. We take $\rho\in(0,1)$ such that $c\rho^{q\theta}-C\rho^{r_-}>0$, which is possible since $q\theta<r_-$ (see \eqref{superlinear}). Then there exists $\sigma>0$ such that
	\begin{align*}
		\check{\Lambda} := \frac{c\rho^{q\theta}-C\rho^{r_-}-\sigma}{C\rho^{1-\gamma_+}} > 0.
	\end{align*}
	Thus, choosing $\overline{t}:=\frac{\rho}{\|\nabla v\|_\H}$, for all $\lambda<\check{\Lambda}$ one has $J(\overline{t}v)\geq \sigma$. Notice that $\sigma$ is independent of $u$. Since $\lambda<\Lambda_2$, Lemma \ref{fiberstruct} ensures that $\psi_v$ has a unique global maximizer at $t=1$. Hence
	\begin{align*}
		J(v) \geq J(\overline{t}v)\geq \sigma>0
	\end{align*}
	whenever $\lambda<\Lambda_3:=\min\{\Lambda_2,\check{\Lambda}\}$.
\end{proof}

%********************************************************************
\section{Proof of the main result}\label{sec5}
%********************************************************************

We set $\Lambda:=\min\{\Lambda_1,\Lambda_2,\Lambda_3\}$ with $\Lambda_i$, $i=1,2,3$, defined in the Lemmas \ref{N0empty}, \ref{fiberstruct}, and \ref{energysign}, respectively.

\begin{proposition}\label{Ne+min}
	Let hypotheses \eqref{H} be satisfied and let $\lambda\in(0,\Lambda)$. Then there exists $u\in\Ne^+$ such that $u\geq 0$ a.e.\,in $\Omega$ and
	\begin{align*}
		J(u) = \min_{\Ne^+} J.
	\end{align*}
\end{proposition}

\begin{proof}
	Let $\{u_n\}_{n\in\N}$ be a minimizing sequence of $J|_{\Ne^+}$. The coercivity of $J|_{\Ne^+}$ (see Lemma \ref{coercivity}) forces $u_n\rightharpoonup u$ for some $u\in W^{1,\H}_0(\Omega)$, passing to a sub-sequence if necessary. We may assume also $u_n\to u$ in $L^\kappa(\Omega)$ for all $\kappa\in(1,p^*)$. The weak sequential lower semi-continuity of $J$, along with Lemma \ref{energysign}, implies $J(u)\leq\inf_{\Ne^+} J<0$ and so $u\neq 0$. Owing to Lemma \ref{fiberstruct}, there exists a unique $\overline{t}>0$ such that $\overline{t}u\in\Ne^+$. It remains to prove that $u\in\Ne^+$.

	Reasoning as in the first part of the proof of Lemma \ref{energysign}, $\psi_u$ is strictly decreasing in $(0,\overline{t})$.

	We claim that $u_n\to u$ in $W^{1,\H}_0(\Omega)$, up to sub-sequences. The claim is equivalent to $\overline{t}u_n\to \overline{t}u$ in $W^{1,\H}_0(\Omega)$. We argue by contradiction, assuming that $\{\overline{t}u_n\}_{n\in\mathbb{N}}$ does not converge to $\overline{t}u$. We have
	\begin{align*}
		\limsup_{n\to\infty} \phi(\overline{t}u_n) > \phi(\overline{t}u),
	\end{align*}
	since the opposite inequality entails $\overline{t}u_n\to\overline{t}u$ by the uniform convexity of $W^{1,\H}_0(\Omega)$. Moreover, according to the convexity of $\H(x,\cdot)$ for a.a.\,$x\in\Omega$, we get
	\begin{align*}
		0
		&\leq \liminf_{n\to\infty} \langle \phi'(\overline{t}\nabla u_n)-\phi'(\overline{t}\nabla u), \nabla u_n-\nabla u \rangle\\
		&= \liminf_{n\to\infty}\langle \phi'(\overline{t}\nabla u_n),\nabla u_n \rangle - \langle \phi'(\overline{t}\nabla u),\nabla u \rangle.
	\end{align*}
	We deduce
	\begin{align*}
		\limsup_{n\to\infty} \langle \phi'(\overline{t}\nabla u_n),\nabla u_n \rangle > \langle \phi'(\overline{t}\nabla u),\nabla u \rangle.
	\end{align*}
	Indeed, if $\langle \phi'(\overline{t}\nabla u_n),\nabla u_n \rangle \to \langle \phi'(\overline{t}\nabla u),\nabla u \rangle$, then the weak-weak continuity of $\L\colon W^{1,\H}_0(\Omega)\to W^{1,\H}_0(\Omega)^*$ yields
	\begin{align*}
		\limsup_{n\to\infty} \langle \L(\overline{t}u_n),\overline{t}u_n-\overline{t}u\rangle &= \limsup_{n\to\infty} \langle \L(\overline{t}u_n),\overline{t}u_n\rangle - \langle \L(\overline{t}u),\overline{t}u \rangle \\
		&= \overline{t} \left[\lim_{n\to\infty} \langle \phi'(\overline{t}\nabla u_n),\nabla u_n \rangle - \langle \phi'(\overline{t}\nabla u),\nabla u\rangle\right] = 0,
	\end{align*}
	which forces $\overline{t}u_n\to\overline{t}u$ due to the ${\rm (S_+)}$-property of $\L$, ensured by Lemma \ref{S+}. Hence, passing to a sub-sequence, we can assume
	\begin{align}\label{absurdhyp}
		\lim_{n\to\infty} \langle \phi'(\overline{t}\nabla u_n),\nabla u_n \rangle > \langle \phi'(\overline{t}\nabla u),\nabla u \rangle \quad \text{and} \quad \lim_{n\to\infty} \phi(\overline{t}u_n) > \phi(\overline{t}u).
	\end{align}

	Exploiting Lebesgue's dominated convergence theorem, the monotonicity of $m$, and \eqref{absurdhyp}, we infer
	\begin{equation}\label{nostrongconv}
		\begin{aligned}
			&\liminf_{n\to\infty}\psi_{u_n}'(\overline{t}) \\
			&= \liminf_{n\to\infty} \bigg[ m(\phi(\overline{t} \nabla u_n)) \langle \phi'(\overline{t}\nabla u_n),\nabla u_n \rangle - \int_\Omega f(\overline{t}u_n)u_n \dx\\
			&\qquad\qquad\quad- \int_\Omega g(\overline{t}u_n) u_n \dx \bigg] \\
			&= \liminf_{n\to\infty} \left[ m(\phi(\overline{t} \nabla u_n)) \langle \phi'(\overline{t}\nabla u_n),\nabla u_n \rangle \right] - \int_\Omega f(\overline{t}u)u \dx - \int_\Omega g(\overline{t}u)u \dx \\
			&> m(\phi(\overline{t}\nabla u)) \langle \phi'(\overline{t}\nabla u),\nabla u \rangle  - \int_\Omega f(\overline{t}u)u \dx - \int_\Omega g(\overline{t}u)u \dx = \psi_u'(\overline{t})=0,
		\end{aligned}
	\end{equation}
	which forces $\overline{t}>1$. Indeed $\psi_{u_n}'(1)=0$. Reasoning as in the first part of Lemma \ref{energysign}, from $E_u^+=\{\overline{t}\}$ and $E_{u_n}^+=\{1\}$ we deduce $\psi_u'<0$ in $(0,\overline{t})$ and $\psi_{u_n}'<0$ in $(0,1)$. Hence, exploiting also the weak sequential lower semi-continuity of $J$ yields
	\begin{align*}
		\inf_{\Ne^+} J \leq J(\overline{t}u) = \psi_u(\overline{t}) < \psi_u(1) = J(u) \leq \liminf_{n\to\infty} J(u_n) = \inf_{\Ne^+} J,
	\end{align*}
	which is a contradiction. This establishes $u_n\to u$ in $W^{1,\H}_0(\Omega)$ up to sub-sequences, as claimed.

	Letting $n\to\infty$ in both $\psi_{u_n}'(1)=0$ and $\psi_{u_n}^{''}(1)>0$, besides recalling Remark \ref{differentiability}, we get $u\in\Ne^+\cup \Ne^0$. Taking into account Lemma \ref{N0empty}, we deduce $u\in\Ne^+$. By the symmetry of $J$ and $\Ne^+$, one can replace $u$ with $|u|$, so that it is possible to assume $u\geq 0$ a.e.\,in $\Omega$.
\end{proof}

\begin{proposition}\label{Ne-min}
	Let hypotheses \eqref{H} be satisfied and let $\lambda\in(0,\Lambda)$. Then there exists $u\in\Ne^-$ such that $u\geq 0$ a.e.\,in $\Omega$ and
	\begin{align*}
		J(u) = \min_{\Ne^-} J.
	\end{align*}
\end{proposition}

\begin{proof}
	Take any minimizing sequence $\{u_n\}_{n\in\N}\subseteq\Ne^-$  for $J_{\mid_{\Ne^-}}$. The proof is analogous to the one of Proposition \ref{Ne+min}, except the non-triviality of $u$ (that is, the weak limit of $\{u_n\}_{n\in\N}$ in $W^{1,\H}_0(\Omega)$) and the strong convergence of $\{u_n\}_{n\in\N}$ in $W^{1,\H}_0(\Omega)$.

	In order to prove that $u\neq 0$ we argue by contradiction, supposing that $u_n\rightharpoonup 0$ in $W^{1,\H}_0(\Omega)$. Without any loss of generality, $u_n\to 0$ in $L^\kappa(\Omega)$ for all $\kappa\in(1,p^*)$. Since $u_n\in\Ne^-$ for all $n\in\N$, we have
	\begin{align*}
		m(\phi(\nabla u_n))\langle \phi'(\nabla u_n),\nabla u_n \rangle = \lambda\int_\Omega f(u_n)u_n \dx + \int_\Omega g(u_n)u_n \dx \quad \text{for all } n\in\N.
	\end{align*}
	Letting $n\to\infty$, along with Lemma \ref{basicest1}, reveals
	\begin{align*}
		\lim_{n\to\infty} \underline{W}_p^{q\theta}(\|\nabla u_n\|_\H) \leq C\lim_{n\to\infty} m(\phi(\nabla u_n))\langle \phi'(\nabla u_n),\nabla u_n \rangle = 0
	\end{align*}
	for some $C>0$, which entails $u_n\to 0$ in $W^{1,\H}_0(\Omega)$. According to Lemma \ref{energysign},
	\begin{align*}
		0 = J(0) = \lim_{n\to\infty} J(u_n) \geq \sigma,
	\end{align*}
	which is a contradiction.

	Now we prove $u_n\to u$ in $W^{1,\H}_0(\Omega)$. Since $u\neq 0$, Lemma \ref{fiberstruct} produces a unique $\overline{t}\in(0,+\infty)$ such that $\overline{t}u\in\Ne^-$. Reasoning by contradiction as in Proposition \ref{Ne+min}, namely supposing \eqref{absurdhyp}, we deduce $\psi_{u_n}'(\overline{t})>0$ for $n$ sufficiently large (cf.\,\eqref{nostrongconv}). This forces $\overline{t}<1$, taking into account that, for any $n\in\N$, one has $E_{u_n}^{-}=\{1\}$ and $\psi_{u_n}'(t)<0$ for all $t>1$. Moreover, $t=1$ is the unique global maximizer of $\psi_{u_n}$. Indeed, it is the unique local maximizer, and $\psi'_{u_n}(t)<0$ for all $t>1$ as well as $\psi_{u_n}(1)=J(u_n)\geq \sigma>0=J(0)$, due to Lemmas \ref{fiberstruct} and \ref{energysign}, respectively. This information, together with \eqref{absurdhyp} and the strict monotonicity of $M$, yields
	\begin{align*}
		\inf_{\Ne^-} J
		&\leq J(\overline{t}u) < \liminf_{n\to\infty} J(\overline{t}u_n) = \liminf_{n\to\infty} \psi_{u_n}(\overline{t})\\
		&\leq \liminf_{n\to\infty}\psi_{u_n}(1) = \liminf_{n\to\infty} J(u_n) = \inf_{\Ne^-} J,
	\end{align*}
	which is a contradiction.
\end{proof}

\begin{proposition}\label{locmin}
	Let hypotheses \eqref{H} be satisfied  and let $u\in \Ne^+$ be such that $J(u)=\min_{\Ne^+} J$. Then there exists $\eps>0$ such that
	\begin{align*}
		J(u)\leq J(u+h) \quad \text{for all } h\in B_\eps(0).
	\end{align*}
\end{proposition}

\begin{proof}
	Let us consider the function
	\begin{align*}
		F(h,t) = \psi_{u+h}'(t) \quad \text{for all } (h,t)\in W^{1,\H}_0(\Omega)\times(0,+\infty).
	\end{align*}
	Since $u\in\Ne^+$ one has $F(0,1)=\psi_u'(1)=0$ and $\partial_t F(0,1) = \psi_u''(1)>0$. Hence the implicit function theorem furnishes $\eps_1>0$ and $\zeta\colon  B_{\eps_1}(0)\to(0,+\infty)$ such that $\zeta(0)=1$ and $F(h,\zeta(h))=0$, that is, $\zeta(h)(u+h)\in\Ne$ by \eqref{psiscaling}. According to Remark \ref{differentiability}, $\partial_t F$ is a continuous function. Thus there exist $\eps_2,\sigma>0$ such that
	\begin{equation}\label{convexity}
		\psi_{u+h}''(t)=\partial_t F(h,t)>0 \quad \text{for all } (h,t)\in B_{\eps_2}(0)\times (1-\sigma,1+\sigma).
	\end{equation}
	The function $\zeta$ is continuous as well, so there exists $\eps_3>0$ such that $\zeta(h)\in(1-\sigma,1+\sigma)$ for all $h\in B_{\eps_3}(0)$. Setting $\eps=\min\{\eps_1,\eps_2,\eps_3\}$, we deduce that $\zeta(h)(u+h)\in \Ne^+$ for all $h\in B_\eps(0)$. In particular, \eqref{convexity} implies also the convexity of $\psi_{u+h}$ in the interval joining $t=\zeta(h)$ and $t=1$. Hence, we have
	\begin{align*}
		\psi_{u+h}(\zeta(h))\leq \psi_{u+h}'(\zeta(h))(\zeta(h)-1)+\psi_{u+h}(1)=\psi_{u+h}(1).
	\end{align*}
	Accordingly,
	\begin{align*}
		J(u)=\min_{\Ne^+} J \leq J(\zeta(h)(u+h)) = \psi_{u+h}(\zeta(h)) \leq \psi_{u+h}(1) = J(u+h)
	\end{align*}
	for all $h\in B_\eps(0)$.
\end{proof}

\begin{remark}\label{saddlepoint}
	The conclusion of Proposition \ref{locmin} does not hold for the minimizers of $J$ constrained to $\Ne^-$ because they are not local minimizers of $J$ on $W^{1,\H}_0(\Omega)$. Instead they are saddle points. Indeed, given any $u$ such that $J(u)=\min_{\Ne^-} J$, $u$ is a strict local maximizer along the direction of $u$, while (reasoning as in Proposition \ref{locmin}, that furnishes $\zeta$ such that $\zeta(th)(u+th)\in\Ne^-$ for small $t$) it is a local minimizer along any curve of type $t\mapsto \zeta(th)(u+th)$ with $h\in W^{1,\H}_0(\Omega)\setminus\{0\}$.
\end{remark}

\begin{lemma}\label{fundineq}
	Let hypotheses \eqref{H} be satisfied and let $u\in W^{1,\H}_0(\Omega)$, $u\geq 0$ a.e.\,in $\Omega$, be a local minimizer of $J$. Then $u>0$ a.e.\,in $\Omega$ and
	\begin{equation}\label{weakineq}
		m(\phi(\nabla u))\langle \L(u),h\rangle \geq \lambda \int_\Omega f(u)h \dx + \int_\Omega g(u)h \dx
	\end{equation}
	for all $h\in W^{1,\H}_0(\Omega)$ with $h\geq 0$ a.e.\,in $\Omega$.
\end{lemma}

\begin{proof}
	Take any $h\in W^{1,\H}_0(\Omega)\setminus\{0\}$ fulfilling $h\geq 0$ a.e.\,in $\Omega$. Since $u$ is a local minimizer of $J$, then $J(u)\leq J(u+th)$ for all $t$ sufficiently small. Take any sequence $\{t_n\}_{n\in\N}$ with $t_n>0$ for all $n\in\N$ such that $t_n\to 0$ and set $K=u^{-1}(0)$. Then, for any $n$ large enough,
	\begin{align*}
		0&\leq \frac{J(u+t_nh)-J(u)}{t_n} \\
		&= \frac{M(\phi(\nabla (u+t_nh)))-M(\phi(\nabla u))} {t_n} - \lambda \int_\Omega \frac{F(u+t_nh)-F(u)}{t_n} \dx \\
		&\quad - \int_\Omega \frac{G(u+t_nh)-G(u)}{t_n} \dx \\
		&=\frac{M(\phi(\nabla (u+t_nh)))-M(\phi(\nabla u))} {t_n} - \lambda \int_K \frac{F(t_nh)}{t_n} \dx \\
		&\quad - \lambda \int_{\Omega\setminus K} \frac{F(u+t_nh)-F(u)}{t_n}  \dx - \int_\Omega \frac{G(u+t_nh)-G(u)}{t_n} \dx.
	\end{align*}
	From Lebesgue's dominated convergence theorem we obtain
	\begin{align*}
		\lim_{n\to\infty} \frac{M(\phi(\nabla (u+t_nh)))-M(\phi(\nabla u))}{t_n} &= m(\phi(\nabla u))\langle \L(u),h\rangle,\\
		\lim_{n\to\infty} \int_\Omega \frac{G(u+t_nh)-G(u)}{t_n} \dx &= \int_\Omega g(u)h \dx,
	\end{align*}
	while Fatou's lemma and the monotonicity of $F$ yields
	\begin{align*}
		\liminf_{n\to\infty} \int_{\Omega\setminus K} \frac{F(u+t_nh)-F(u)}{t_n} \dx \geq \int_{\Omega\setminus K} f(u)h \dx.
	\end{align*}
	Accordingly,
	\begin{equation}\label{fundineq1}
		\begin{aligned}
			0&\leq \limsup_{n\to\infty} \frac{J(u+t_nh)- J(u)}{t_n} \\
			&\leq m(\phi(\nabla u))\langle \L(u),h\rangle -  \lambda \int_{\Omega\setminus K} f(u)h \dx\\
			&\quad - \int_\Omega g(u)h \dx - \lambda \liminf_{n\to\infty} \int_K \frac{F(t_nh)}{t_n} \dx.
		\end{aligned}
	\end{equation}
	If $K$ has positive measure, then \eqref{hypf} forces
	\begin{align*}
		\lim_{n\to\infty} \int_K \frac{F(t_nh)}{t_n} \dx = +\infty,
	\end{align*}
	which is a contradiction. Hence $K$ has zero measure, that is, $u>0$ a.e.\,in $\Omega$. So \eqref{fundineq1} rewrites as
	\begin{align*}
		0\leq m(\phi(\nabla u))\langle \L(u),h\rangle - \lambda \int_\Omega f(u)h \dx - \int_\Omega g(u)h \dx.
	\end{align*}
	This inequality is obviously verified also for $h=0$, which concludes the proof.
\end{proof}

\begin{lemma}\label{fundineq-}
	Let hypotheses \eqref{H} be satisfied and let $u\in \Ne^-$ be such that $u\geq 0$ a.e.\,in $\Omega$ and $J(u)=\min_{\Ne^-} J$. Then $u>0$ a.e.\,in $\Omega$ and fulfills \eqref{weakineq}.
\end{lemma}

\begin{proof}
	We only sketch this proof, which is similar to those of Proposition \ref{locmin} and Lemma \ref{fundineq}.

	Reasoning as in Proposition \ref{locmin}, there exists $\eps>0$ and a continuous function $\zeta\colon B_\eps(0)\to(0,+\infty)$ such that $\zeta(0)=1$ and
	\begin{align*}
		\zeta(h)(u+h)\in \Ne^- \quad \text{for all } h\in B_\eps(0).
	\end{align*}
	In particular, owing to $u\in \Ne^-$, one has $\psi_u(\zeta(th))\leq \psi_u(1)$ for all $h\in W^{1,\H}_0(\Omega)$ and $t$ sufficiently small. Take any sequence $\{t_n\}_{n\in\N}$ with $t_n>0$ for all $n\in\N$ such that $t_n\to 0$ and set $K=u^{-1}(0)$. For any $n$ sufficiently large we get
	\begin{equation}\label{fermat}
		\begin{aligned}
			0
			&\leq \frac{J(\zeta(t_nh)(u+t_nh))-J(u)}{t_n} \leq \frac{J(\zeta(t_nh)(u+t_nh))- J(\zeta(t_nh)u)}{t_n} \\
			% &= \frac{M(\phi(\nabla [\zeta(t_nh)(u+t_nh)]))- M(\phi(\nabla [\zeta(t_nh)u]))}{t_n} \\
			% &\quad -  \lambda \int_\Omega \frac{F(\zeta(t_nh)(u+t_nh))-F(\zeta(t_nh)u)}{t_n} \dx \\
			% &\quad - \int_\Omega \frac{G(\zeta(t_nh)(u+t_nh))- G(\zeta(t_nh)u)}{t_n} \dx \\
			&=\frac{M(\phi(\nabla [\zeta(t_nh)(u+t_nh)]))- M(\phi(\nabla [\zeta(t_nh)u]))}{t_n}\\
			&\quad- \lambda \int_K  \frac{F(\zeta(t_nh)t_nh)}{t_n} \dx \\
			&\quad - \lambda \int_{\Omega \setminus K} \frac{F(\zeta(t_nh)(u+t_nh))-F(\zeta(t_nh)u)}{t_n} \dx \\
			&\quad - \int_\Omega \frac{G(\zeta(t_nh)(u+t_nh))- G(\zeta(t_nh)u)}{t_n} \dx.
		\end{aligned}
	\end{equation}
	Fix any $t>0$ and consider the function
	\begin{align*}
		\Gamma\colon[0,t]\to\R, \quad \Gamma(s) := M(\phi(\nabla [\zeta(th)(u+sh)])).
	\end{align*}
	Lagrange's mean value theorem produces $s_t\in(0,t)$ such that
	\begin{align*}
		&M(\phi(\nabla [\zeta(th)(u+th)]))-M(\phi(\nabla [\zeta(th)u])) \\
		&= tm(\phi(\nabla [\zeta(th)(u+s_th)]))\langle\phi'(\nabla [\zeta(th)(u+s_th)]),\nabla[\zeta(th)h]\rangle.
	\end{align*}
	Hence, recalling also $\zeta(th)\to 1$ as $t\to 0^+$,
	\begin{align*}
		&\lim_{t\to 0^+} \frac{1}{t} \left[M(\phi(\nabla [\zeta(th)(u+th)]))- M(\phi(\nabla [\zeta(th)u]))\right] \\
		&= \lim_{t\to 0^+}  m(\phi(\nabla [\zeta(th)(u+s_th)]))\langle\phi'(\nabla [\zeta(th)(u+s_th)]),\nabla [\zeta(th)h]\rangle \\
		&= m(\phi(\nabla u))\langle\phi'(\nabla u), \nabla h\rangle = m(\phi(\nabla u))\langle \L(u),h\rangle.
	\end{align*}
	Arguing in the same way for the difference quotients involving $F$ and $G$, \eqref{fermat} yields
	\begin{align*}
		0
		&\leq \limsup_{n\to\infty} \frac{J(\zeta(t_nh)(u+t_nh))-J(u)}{t_n} \\
		&\leq m(\phi(\nabla u))\langle \L(u),h\rangle - \lambda \int_{\Omega\setminus K} f(u)h \dx - \int_\Omega g(u)h \dx \\
		&\quad -\lambda \liminf_{n\to\infty} \int_K \frac{F(\zeta(t_nh)t_nh)}{t_n} \dx,
	\end{align*}
	which parallels \eqref{fundineq1}. The proof now follows exactly as in Proposition \ref{fundineq}.
\end{proof}

\begin{proposition}\label{weaksol}
	Let hypotheses \eqref{H} be satisfied. Any $u\in \Ne$ satisfying both $u>0$ a.e.\,in $\Omega$ and \eqref{weakineq} is a weak solution to \eqref{prob}.
\end{proposition}

\begin{proof}
	Let us consider the linear operator $T\colon W^{1,\H}_0(\Omega)\to\R$ defined as
	\begin{align*}
		\langle T,h \rangle = m(\phi(\nabla u))\langle \L(u),h\rangle - \lambda \int_\Omega f(u)h \dx - \int_\Omega g(u)h \dx \quad \text{for all } h\in W^{1,\H}_0(\Omega).
	\end{align*}
	According to \eqref{weakineq}, $T$ is well-defined and non-negative (i.e., $\langle T,h \rangle \geq 0$ for all $h\geq 0$ a.e.\,in $\Omega$). Moreover, $u\in\Ne$ is equivalent to $\langle T,u \rangle = 0$. Hence, taking any $\varphi\in W^{1,\H}_0(\Omega)$ and $\eps>0$, we have
	\begin{align*}
		0&\leq \langle T,(u+\eps\varphi)^+ \rangle\\
		&= \langle T,u+\eps\varphi  \rangle + \langle T,(u+\eps\varphi)^- \rangle \\
		&= \langle T,u \rangle + \eps  \langle T,\varphi \rangle + \langle T,(u+\eps\varphi)^- \rangle \\
		&= \eps \langle T,\varphi \rangle  + \langle T,(u+\eps\varphi)^- \rangle,
	\end{align*}
	where $(u+\eps\varphi)^+$ and $(u+\eps\varphi)^-$ stand for the positive and the negative part of $u+\eps\varphi$, respectively. Recalling the definition of $T$, $u>0$ a.e.\,in $\Omega$, and $\partial_s\H(x,|\nabla u|)\geq 0$ a.e.\,in $\Omega$, we have
	\begin{align*}
		\langle T,(u+\eps\varphi)^- \rangle
		&\leq m(\phi(\nabla u)) \langle  \L(u),(u+\eps\varphi)^- \rangle \\
		&= -m(\phi(\nabla u))  \int_{\{u+\eps\varphi\leq 0\}} \partial_s\H(x,|\nabla u|)\frac{\nabla u}{|\nabla u|}(\nabla u+\eps\nabla\varphi) \dx \\
		&\leq -\eps m(\phi(\nabla u))  \int_{\{u+\eps\varphi\leq 0\}} \partial_s\H(x,|\nabla u|)\frac{\nabla u}{|\nabla u|}\nabla\varphi \dx.
	\end{align*}
	Thus we get
	\begin{align*}
		0 \leq \langle T,\varphi \rangle - m(\phi(\nabla u)) \int_{\{u+\eps\varphi\leq 0\}} \partial_s\H(x,|\nabla u|)\frac{\nabla u}{|\nabla u|}\nabla\varphi \dx.
	\end{align*}
	Notice that $|\{u+\eps\varphi\leq 0\}|\to 0$ as $\eps\to 0$. Therefore, $\langle T,\varphi \rangle \geq 0 \quad \text{for all } \varphi\in W^{1,\H}_0(\Omega)$. Since $\varphi$ is arbitrarily chosen, we have $\langle T,\varphi \rangle = 0$ for all $\varphi\in W^{1,\H}_0(\Omega)$, which means that $u$ is a weak solution to \eqref{prob}.
\end{proof}

Now we can give the proof of our main result.

\begin{proof}[Proof of Theorem \ref{mainthm}]
	Owing to Propositions \ref{Ne+min} and \ref{Ne-min}, we can find functions $u,v\in W^{1,\H}_0(\Omega)$ such that
	\begin{align*}
		J(u) = \min_{\Ne^+} J \quad\text{and}\quad J(v) = \min_{\Ne^-} J.
	\end{align*}
	By virtue of Lemma \ref{locmin} (see also Remark \ref{saddlepoint}), Lemma \ref{fundineq} is applicable to $u$. Thus, Proposition \ref{weaksol} ensures that $u$ is a weak solution of problem \eqref{prob}. On the other hand, Lemma \ref{fundineq-} and Proposition \ref{weaksol} guarantee that $v$ is a weak solution to \eqref{prob}. The conclusion follows by Lemma \ref{energysign}, since $u\in\Ne^+$ and $v\in\Ne^-$ imply
	\begin{align*}
		J(u)<0<J(v).
	\end{align*}
\end{proof}
%
%\noindent
%{\bf Data availability statement.} No new data were created or analysed in this study. Data sharing is not applicable to this article. \\
%{\bf Conflict of interest.} Authors state no conflict of interest.

%********************************************************************
\section*{Acknowledgments}
%********************************************************************

The first author is member of the {\em Gruppo Nazionale per l'Analisi Matematica, la Probabilit\`a e le loro Applicazioni}
(GNAMPA) of the {\em Istituto Nazionale di Alta Matematica} (INdAM).

This study was partly funded by: (i) Research project of MIUR (Italian Ministry of Education, University and Research) Prin 2022 {\it Nonlinear differential problems with applications to real phenomena} (Grant No. 2022ZXZTN2); (ii) INdAM-GNAMPA Project 2023 titled {\em Problemi ellittici e parabolici con termini di reazione singolari e convettivi} (E53C22001930001).

The first author thanks the University of Technology Berlin for the kind hospitality during a research stay in April 2024.

%********************************************************************

\end{document}